\documentclass{amsart}

\usepackage{amsmath}
\usepackage{graphicx}
\usepackage{amsfonts}
\usepackage{amssymb}

\newtheorem{theorem}{Theorem}
\theoremstyle{plain}

\newtheorem{corollary}{Corollary}

\newtheorem{proposition}{Proposition}
\newtheorem{remark}{Remark}

\numberwithin{equation}{section}

\begin{document}

\title[A combinatorial model for the free loop]{A combinatorial model for the free loop fibration}

\author{Manuel Rivera}
\address{Manuel Rivera, Department of Mathematics, University of Miami,  1365 Memorial Drive, Coral
Gables, FL 33146 \\and
Departamento de Matem\'aticas, Cinvestav, Av. Instituto Polit\'ecnico Nacional 2508, Col. San Pedro Zacatenco, M\'exico, D.F. CP 07360, M\'exico}
  \email{manuelr@math.miami.edu}
\author{Samson Saneblidze$^\ast$}
\address{Samson Saneblidze, A. Razmadze Mathematical Institute,
I.Javakhishvili Tbilisi State University 6, Tamarashvili  st.,
 Tbilisi 0177, Georgia}
 \email{sane@rmi.ge}

\thanks{$^\ast$This research described in this publication was made possible in part by the grant \\ SRNSF/217614}

\subjclass[2010]{ 55P35, 55U05, 52B05, 18F20}
\keywords{Simplicial sets, free loop space, necklaces, freehedra, (co)Hochschild complex}

\begin{abstract}
We introduce the abstract notion of a closed necklical set in order to describe a functorial combinatorial model of the free loop fibration $\Omega Y\rightarrow \Lambda Y\rightarrow Y$ over the geometric realization $Y=|X|$ of a path connected simplicial set $X.$ In particular, to any path connected simplicial set $X$ we associate a closed necklical set $\widehat{\mathbf{\Lambda}}X$ such that its geometric realization $|\widehat{\mathbf{\Lambda}}X|$, a space built out of glueing ``freehedrical" and ``cubical" cells, is homotopy equivalent to the free loop space $\Lambda Y$ and the differential graded module of chains $C_*(\widehat{\mathbf{\Lambda}}X)$ generalizes the coHochschild chain complex of the chain coalgebra $C_\ast(X).$
\end{abstract}

\maketitle

 \section{Introduction}

Let $Y$ be a connected topological space. The \textit{free loop space} of $Y$ is the topological space $\Lambda Y= Maps(S^1,Y)$ defined as the set of all loops in $Y$, i.e. continuous maps from the circle $S^1$ to $Y$, equipped with the compact open topology. Fixing a point in the circle $* \in S^1$ we obtain a continuous map $\varrho: \Lambda Y \to Y$ defined by evaluating a loop at $*$, namely $\varrho ( \gamma) =\gamma(*)$. The map $\varrho: \Lambda Y \to Y$ is a fibration, which we call the \textit{free loop fibration}, whose fiber at a point $b \in Y$ is the space $\Omega Y= \varrho^{-1}(b)$ consisting of all loops with the base point $b$. Suppose now that $Y$ is the geometric realization of a connected simplicial set $X$. The main goal of this article is to construct a natural combinatorial model for the free loop fibration of $Y$ based on the combinatorics of $X$. In particular, we introduce the notion of a \textit{closed necklical set} in order to describe a combinatorial model $|\widehat{\mathbf{\Lambda}} X|$ for $\Lambda Y$. A closed necklical set is a presheaf over the category of \textit{closed necklaces}. Closed necklaces are simplicial sets of the form $\Delta^{n_0} \vee ... \vee \Delta^{n_p}$ where each $\Delta^{n_i}$ is a standard simplex with $n_0\geq0$, $n_i \geq 1$ for $i=1,...,p$, and the wedges mean that the last vertex of $\Delta^{n_i}$ is identified with the first vertex of $\Delta^{n_{i+1}}$ for $i=0,...,p-1$ and the last vertex of $\Delta^{n_p}$ is identified with the first vertex of $\Delta^{n_0}$.

We build up on constructions and results of \cite{RS} where a combinatorial model for the based loop space on the geometric realization of a simplicial set $X$ is introduced using the notion of a \textit{necklical set}. Necklical sets are presheaves over the category of \textit{necklaces}. We will recall the necessary definitions and statements throughout the text to keep our presentation self contained. Necklaces are simplicial sets defined similarly to closed necklaces except that we do not identify the first and last vertices and all standard simplices in the wedge are required to have dimension at least $1$.

\newpage

 Necklaces were introduced in \cite{DS11} to describe the mapping spaces of certain functor from simplicial sets to simplicial categories, called the \textit{rigidification functor}, which compares two models for $\infty$-categories. Then necklaces inside a path connected simplicial set $X$ based at a vertex $b \in X_0$ were used in \cite{RS} to label the cubical cells of a combinatorial model for the based loop space and in  \cite{Rivera- Zeinalian} this result is packaged in the language of $\infty$-categories. More precisely, in \cite{RS} we describe a topological space $|\widehat{\mathbf{\Omega}} X|$ which is homotopy equivalent to the based loop space on $Y=|X|$ and is obtained as the geometric realization of a necklical set  $\widehat{\mathbf{\Omega}}X$ defined by glueing cubical cells corresponding to necklaces inside $X$ based at $b$. Because of the non-triviality of the fundamental group $\pi_1(X)$ the construction of $\widehat{\mathbf{\Omega}}X$ involves formally inverting $1$-simplices in $X$ in order to obtain the correct homotopy type of the based loop space, which is the reason for the ``hat" symbol in the notation $\widehat{\mathbf{\Omega}}X.$ In \cite{Rivera- Zeinalian} the different but closely related notion of a \textit{cubical set with connections} was used to obtain a similar construction.

In section 2 we recall the necessary definitions and results from \cite{RS} and \cite{Rivera- Zeinalian}, we define closed necklaces, closed necklical sets, and the geometric realization of a closed necklical set as a topological space obtained by glueing cells each of which is a cartesian product of a standard cube and certain polytope $F_n$ called a \textit{freehedra}, originally introduced in \cite{saneFREE}. We proceed with our main construction and result in section 3: associated to any path connected simplicial set $X$ there is a natural closed necklical set $\widehat{\mathbf{\Lambda}} X$ whose geometric realization  $|\widehat{\mathbf{\Lambda}} X|$ is a space homotopy equivalent to the free loop space $\Lambda Y$ on $Y=|X|$.  In fact, we construct a quasi-fibration
   $|\widehat{\mathbf{\Omega}}X| \xrightarrow{\iota} |\widehat{\mathbf{\Lambda}} X| \xrightarrow{\zeta} Y$
modeling the free loop  fibration  $\Omega Y\rightarrow \Lambda Y\rightarrow Y.$
  The proof of our main result only uses basic tools from classical algebraic topology similar to those used in \cite{RS}.

  Finally,  we explore some algebraic consequences in section 4. In \cite{RS} we explained how the algebra of chains on $\widehat{\mathbf{\Omega}} X$, the necklical set model for the based loop space, generalizes Adams' \textit{cobar construction} on the Alexander-Whitney differential graded (dg) coalgebra of chains on $X$. The chains on the closed necklical set  $\widehat{\mathbf{\Lambda}} X$  provides a small chain complex, suitable for calculations, which computes the homology of the free loop space of a path connected, possibly non-simply connected, space. We explain how such a chain complex generalizes the \textit{coHochschild chain complex}
  (\cite{Cartier},\cite{Hess},\cite{J-M.hoh},\cite{Loday1}) of the dg coalgebra of chains on $X$. We finish the paper by explaining how Goodwillie's result relating the free loop space to the Hochschild chain complex of the dg algebra of chains on the based loop space \cite{Goodwillie} follows as a corollary of our main result.

 The constructions described in this article yield small combinatorial and algebraic models for the free loop space of a connected non-simply connected space which we expect to have direct applications in calculating essential invariants in algebraic topology such as: cup and Massey products, higher cohomology operations, and string topology operations.

 \section{Closed necklaces and closed necklical sets}

We start by recalling the definitions of necklaces and necklical sets as introduced in \cite{DS11} and \cite{RS}. Denote by $Set_{\Delta}$ the category of simplicial sets and by $\Delta^{m}\in Set_{\Delta}$ the standard $m$-simplex. A \textit{necklace} is a wedge of standard simplices $$T=\Delta^{n_1} \vee ... \vee \Delta^{n_k} \in Set_{\Delta}$$ where the last vertex of $\Delta^{n_i}$ is identified with the first vertex of $\Delta^{n_{i+1}}$ and $n_i \geq 1$ for $i=1,...,k-1$. Each $\Delta^{n_i}$ is a subsimplicial set of $T$, which we call a \textit{bead} of $T$. Denote by $b(T)$ the number of beads in $T$. The set $T_0$, or the \textit{vertices} of $T$, inherits an ordering from the ordering of the beads in $T$ and the ordering of the vertices of each $\Delta^{n_i}$. A morphism of necklaces $f: T \to T'$  is a morphism of simplicial sets which preserves first and last vertices. If $T= \Delta^{n_1} \vee ... \vee \Delta^{n_k}$ is a necklace then the dimension of $T$ is defined to be $\text{dim}(T)= n_1 + \cdots + n_k -k$. Denote by $Nec$ the category of necklaces. A \textit{necklical set} is a functor $ Nec^{op} \rightarrow Set$ and a morphism of necklical sets is given by a natural transformation of functors. Denote the category of necklical sets by $Set_{Nec}$.

For our constructions, it will be useful to use the set of generators of morphisms in $Nec$ listed in the proposition below. For a proof see \cite{Rivera- Zeinalian} (Proposition 3.1).

\begin{proposition} Any non-identity morphism in $Nec$ is a composition of morphisms of the following type
\begin{itemize}
\item [(i)]  $f: T \to T'$ is an injective morphism of necklaces and $ \dim(T') - \dim(T) =1;$

\item [(ii)] $f: \Delta^{n_1} \vee ... \vee \Delta^{n_k} \to \Delta^{m_1} \vee ... \vee \Delta^{m_k}$ is a morphism of necklaces of the form $f=f_1 \vee ... \vee f_k$ such that for exactly one $p$
    with $n_p\geq 2,$
     $f_p=s^j: \Delta^{n_p} \to \Delta^{m_p}$ is a co-degeneracy morphism $s^j$ for some $j$ (so $m_p=n_p-1$) and for all $i \neq p$, $f_i: \Delta^{n_i}  \to \Delta^{m_i}$ is the identity map of standard simplices (so $n_i=m_i$ for $i \neq p$);

\item [(iii)] $f: \Delta^{n_1} \vee ...\vee \Delta^{n_{p-1}} \vee \Delta^{n_p} \vee \Delta^{n_{p+1}} \vee... \vee  \Delta^{n_k} \to \Delta^{n_1} \vee ...\vee \Delta^{n_{p-1}} \vee \Delta^{n_{p+1}} \vee... \vee  \Delta^{n_k}$ is a morphism of necklaces such that $f$ collapses the $p$-th bead $\Delta^{n_p}$  in the domain to the last vertex of the $(p-1)$-th bead in the target and the restriction of $f$ to all the other beads is injective.
\end{itemize}
\end{proposition}

We now define the category $Nec_c$ of \textit{closed necklaces}. The objects of $Nec_c$ are simplicial sets of the form $R=\Delta^{n_0}\vee T$, where $n_0 \geq 0$, $T= \Delta^{n_1} \vee ... \vee \Delta^{n_k}$ is a necklace in $Nec$, and the first vertex of $\Delta^{n_0}$ is identified with the last vertex of $T$.  We will call $\Delta^{n_0}$ and $\Delta^{n_k}$ the \textit{first}  and \textit{last beads} of $R$, respectively. The vertices of $R$ also inherit an natural ordering from the ordering of the set of beads of $R$ and ordering of the vertices on each bead.

Morphisms between closed necklaces are defined to be maps of simplicial sets which preserve first beads (but not necessarily preserve first vertices).  If $R=\Delta^{n_0} \vee T= \Delta^{n_0} \vee \Delta^{n_1} \vee ... \vee \Delta^{n_k}$ is a closed necklace then the dimension of $R$ is defined to be $\dim(R)=n_0+\dim(T)= n_0 + n_1 + \cdots + n_k -k$.

A \textit{closed necklical set} is a functor $K: Nec_c^{op} \rightarrow Set$ and a morphism of closed necklical sets is given by a natural transformation of functors. Denote the category of closed necklical sets by $Set_{Nec_c}.$
A simplicial set $X$   gives rise to an  example of a closed necklical set $K_X: Nec_c^{op} \to Set$
 via  the assignment $K_X(R)= Hom(R,X)$, the set of all simplicial set maps from $R$ to $X$. Now we describe a useful set of generators for the morphisms in $Nec_c$ similar to those described for $Nec$ in Proposition 1.

\begin{proposition}

Any non-identity morphism in $Nec_c$  is a composition of morphisms of the following type:

\begin{itemize}
\item[(i)] injective morphisms $f: R \to R'$ of closed necklaces such that $\dim(R') - \dim(R)=1$ and $f$ preserves the last beads;
\item[(i')] injective morphisms $f:R \to R'$ of closed necklaces such that $\dim(R') - \dim(R)=1$ and $f$ maps the last bead of $R$ into the first bead of $R';$
\item[(ii)] morphisms $f:  \Delta^{n_0} \vee \Delta^{n_1} \vee ... \vee \Delta^{n_k} \to \Delta^{n_0}  \vee \Delta^{m_1} \vee ... \vee \Delta^{m_k}$ where $f= id_{_{\Delta^{n_0}}} \vee f_1 \vee ... \vee f_k$ such that for exactly one $p$, with $p \geq 1$ and
    $n_p\geq 2,$ $f_p=s^j: \Delta^{n_p} \to \Delta^{m_p}$ is a simplicial co-degeneracy morphism for some $j$ (so $m_p=n_p-1$) and for all $i \neq p$, $f_i: \Delta^{n_i}  \to \Delta^{m_i}$ is the identity map of standard simplices (so $n_i=m_i$ for $i \neq p$);
\item[(ii')] morphisms $s^j \vee id_T: \Delta^{n_0+1} \vee  T \to \Delta^{n_0} \vee T$ where $s^j: \Delta^{n_0+1} \to \Delta^{n_0 }$ is a simplicial co-degeneracy morphism for some $1 \leq j \leq n_0+1;$
\item[(iii)] morphisms
 \begin{multline*}
 \hspace{0.5in}
 f: \Delta^{n_0} \vee \Delta^{n_1} \vee ... \vee \Delta^{n_p}  \vee... \vee  \Delta^{n_k} \rightarrow \\
  \Delta^{n_0} \vee \Delta^{n_1}...\vee \Delta^{n_{p-1}} \vee \Delta^{n_{p+1}} \vee... \vee  \Delta^{n_k}
 \end{multline*}
     such that, for some $1 \leq p \leq k$ with $k\geq 2,$  $f$ collapses the $p$-th bead $\Delta^{n_p}$ in the domain to the last vertex of the $(p-1)$-th bead $\Delta^{n_{p-1}}$
      in the target and the restriction of $f$ to all the other beads is injective.
\end{itemize}
\end{proposition}

\begin{proof}
Let $Nec^{\geq 2}$ be the full subcategory of $Nec$ consisting of necklaces with at least two beads. Consider the functor $\mathcal{F}: Nec^{\geq 2} \to Nec_c$ defined on objects by sending a necklace $T$ to the closed necklace $\mathcal{F}(T)$ obtained by identifying the first and last vertices of $T$. Any morphism of necklaces $f: T\to T'$ induces a well defined morphism of closed necklaces $\mathcal{F}(f): \mathcal{F}(T) \to \mathcal{F}(T')$ since, by definition, $f$ preserves first and last vertices. The functor $\mathcal{F}$ is faithful (but not full) and injective on objects (but not surjective). Therefore, those morphisms in $Nec_c$ which are in the image of $\mathcal{F}$ can be described exactly as in Proposition 1 (Proposition 3.1 of \cite{Rivera- Zeinalian}), which tells us that each of these morphisms is a composition of morphisms of type $(i)$, $(ii)$, $(ii')$ or $(iii)$. We have separated $(ii)$ and $(ii')$ into two different types because in first bead in closed necklaces will play a particular role later on; in particular, note that in type $(ii')$ it is possible to have $n_0=0$ while in type $(ii)$, $m_p \geq 1$.

Closed necklaces of the form $\Delta^0 \vee T$ are not in the image of $\mathcal{F}$. However, morphisms of closed necklaces in which the domain or the target is of the form $\Delta^0 \vee T$ are generated by morphisms of type $(i), (ii), (ii')$ or $(iii)$. The argument in the proof of Proposition 1 (see \cite{Rivera- Zeinalian}) applies to this case as well. Finally, to obtain the rest of the morphisms of closed necklaces which are not in the image of $\mathcal{F}$ we can add morphisms of type $(i')$ to the generators. This follows since those injective morphisms $f: R \to R'$ of closed necklaces with $\dim(R') - \dim(R)=1$ which are not in the image of $\mathcal{F}$ are exactly those of type $(i')$ and any non-injective morphism $f: R\to R'$ with $\dim(R)- \dim(R')=1$ must be in the image of $\mathcal{F}$ (in fact it must be of type $(ii), (ii')$ or $(iii)$).
\end{proof}

\begin{remark}\normalfont
 A morphism $f: R= \Delta^{n_0} \vee \Delta^{n_1} \vee ... \vee \Delta^{n_k} \to R'=\Delta^{m_0} \vee \Delta^{m_1} \vee... \vee \Delta^{m_l}$ of type $(i)$ can be further classified into two sub-types:

 \noindent $(ia)$ the number of vertices of $R$ is one less than the number of vertices of $R'$ (in particular, this implies $k=l$), and

 \noindent $(ib)$ the number of vertices of $R$ and $R'$ are equal (which, in particular, implies $l=k+1$).

 \noindent Morphisms of type $(ia)$ are of the form $f= id \vee d^{j} \vee id$ where $d^j: \Delta^{n_i} \to \Delta^{m_i}$ for $j \in \{1,...,n_{i}-1\}$ is the simplicial $j$-th co-face morphism for some for some $i\in \{0,...,k\}$ with $n_i+1=m_i$. Morphisms $f$ of type $(ib)$ are those for which there are $i \in \{1,...,k \}$ and $j \in \{1,...,m_i-1\}$ such that $f=id \vee W^j \vee id$, where $W^j: \Delta^{n_i} \vee \Delta^{n_{i+1}} \to  \Delta^{m_i}$ for $n_i + n_{i+1}=m_i$, $i \in \{0,...,k-1\}$, is the injective map of simplicial sets whose image in $\Delta^{m_i}$ is the wedge of the two sub-simplicial sets corresponding to the $j$-th term in the Alexander-Whitney diagonal map applied to the unique non-degenerate top dimensional simplex in $\Delta^{m_i}$.

 \noindent Morphisms of type $(i')$ are necessarily of the following form: injective maps $f: R= \Delta^{n_0} \vee \Delta^{n_1} \vee ... \vee \Delta^{n_k} \to R'=\Delta^{m_0} \vee \Delta^{m_1} \vee... \vee \Delta^{m_{k-1}}$ such that $\dim(R') - \dim(R)=1$, $R$ and $R'$ have the same number of vertices, $n_k +n_0= m_0$, and the restriction of $f$ to the last and first beads of $R$ yields a map $W^j: \Delta^{n_k} \vee \Delta^{n_0} \to \Delta^{m_0}$ corresponding to the $j$-th term in the Alexander-Whitney diagonal map applied to the unique non-degenerate top dimensional simplex of $\Delta^{m_0}.$
\noindent
For each closed necklace $R$ of dimension $n$ there are exactly $n-1$ morphisms $\delta^{i}_1: R^1_i \to R$ $(i=2,...,n)$ of type $(ia)$ and $n$ morphisms $\delta_0^{i}: R^0_i \to R$ $(i=1,...,n)$ of type $(ib)$. Moreover, given a closed necklace $R= \Delta^{n_0} \vee T,$  there are exactly $n_0$ morphisms of closed necklaces
    \begin{equation}\label{ad2}
    \delta^{i}_2: \Delta^{k}\vee T \vee \Delta^{i} \to \Delta^{n_0} \vee T    , \ \  i=1,...,n_0,
    \end{equation}
    of type $(i')$  such that $k+ i = n_0$ for $k \geq 0.$  By convention we define $\delta^1_1:= \delta^1_2$.
\end{remark}

 \section{Closed necklical model for the free loop space}

For any simplicial set $X$ consider the graded set \[ \Big(\bigsqcup_{R \in Nec_c} Hom(R,X) \Big)/\sim\]  where $\sim$ is the equivalence relation generated by the following rules:

\begin{eqnarray}\label{rule1}
f \circ s^{n_p+1} \sim f\circ  s^{0},  \, 1\leq p \leq k,
\end{eqnarray}
for any $f: \Delta^{n_0} \vee \Delta^{n_1} \vee ... \vee \Delta^{n_k} \rightarrow  X$ where
\begin{multline*}
s^{n_p+1}: \Delta^{n_0} \vee \Delta^{n_1} \vee ...\vee \Delta^{n_{p-1}} \vee \Delta^{n_p+1} \vee  \Delta^{n_{p+1}}\vee... \vee  \Delta^{n_k} \to \\
\Delta^{n_0} \vee \Delta^{n_1} \vee ... \vee \Delta^{n_{p-1}} \vee \Delta^{n_p} \vee \Delta^{n_{p+1}} \vee ...\vee \Delta^{n_k}
\end{multline*} is given by applying the last co-degeneracy map to the $p$-th bead, and
\begin{multline*}
s_0: \Delta^{n_0} \vee \Delta^{n_1} \vee ...\vee \Delta^{n_{p-1}} \vee \Delta^{n_p} \vee \Delta^{n_{p+1}+1}\vee \Delta^{n_{p+2}} \vee... \vee  \Delta^{n_k} \to \\
 \Delta^{n_0} \vee \Delta^{n_1} \vee ... \vee \Delta^{n_p} \vee \Delta^{n_{p+1}}\vee \Delta^{n_{p+2}} \vee ...\vee \Delta^{n_k}
\end{multline*}
by applying the first co-degeneracy map to the $(p+1)$-th bead (if $p=k$ then $p+1$ should be interpreted as $0$); and
\begin{equation}\label{rule2}
f \circ u \sim f,
\end{equation}
for any $f \in  Hom(R,X)$ and any morphism $u$ in $Nec$ of type $(iii)$. Denote  the $\sim$-equivalence class of $(f: R\to X)$ by
 $[f: R \to X]$.

\subsection{The closed necklical set $\mathbf{\Lambda} X$}
For any simplicial set $X$ define a closed necklical set $\mathbf{\Lambda}(X): Nec_c^{op} \to Set$ by declaring $\mathbf{\Lambda}(X)(R)$ to be the subset of $\Big(\bigsqcup_{R' \in Nec_c} Hom(R',X) \Big)/\sim$  consisting of all $\sim$-equivalence classes represented by morphisms $R \to X \in Nec_c \downarrow X$. This clearly defines a functor: given a morphism $u: R \to R'$ in $Nec_c$ and an element $[f: R' \to X] \in \mathbf{\Lambda} (X) (R')$ we obtain a well defined element  $ [f \circ u: R \to R' \to X ] \in \mathbf{\Lambda}(X)(R)$.  In particular, $\mathbf{\Lambda} X=\{\mathbf{\Lambda}_{n_0,r,k} X\}_{n_0,r\geq 0,k\geq 1}$ is a trigraded set with
$\mathbf{\Lambda}_{n_0,r,k} X:= \{\Delta^{n_0}\vee T \to X \in (Nec_c \downarrow X)\mid \dim T=r,\, b(T)=k  \} / \sim .$
Note that $\mathbf{\Lambda} (X)$ is precisely the following colimit in the category of necklical sets
\begin{eqnarray*}
\mathbf{\Lambda} (X)=  \underset{f: R \to X \in (Nec \downarrow X) } {\text{colim}} Y(R),
\end{eqnarray*}
where $Y\! : Nec_c \to Set_{Nec_c}$ denotes the Yoneda embedding $Y(R)=\! Hom_{Nec_c}( -\,, R)$. The closed necklical set $\mathbf{\Lambda} (X)$ may be thought of as a free base point analogue of the necklical set $\mathbf{\Omega}(X;x)$, which was defined in \cite{RS} as
\begin{eqnarray*}
\mathbf{\Omega}(X;x)= \underset{ f: T \to X \in (Nec \downarrow X)_x} {\text{colim}} Y(T)
\end{eqnarray*}
where $x \in X_0$ is a fixed point, $(Nec \downarrow X)_x$ denotes the category of maps $f:T \to X$ from some $T \in Nec$ such that $f$ sends the first and last vertices of $T$ to $x$, and $Y(T)= Hom_{Nec}(-\, , T)$.

  Define the \textit{closed necklical face maps}
\[
\begin{array}{ll}
d^1_i : \mathbf {\Lambda} _{n_0,r,k}X \rightarrow \mathbf {\Lambda}_{n_0-1,r,k}X & \text{for}\ \   1\leq i\leq n_0,\vspace{1mm}\\
d^1_i : \mathbf {\Lambda} _{n_0,r,k}X \rightarrow \mathbf {\Lambda}_{n_0,r-1,k}X  & \text{for}\ \   n_0< i\leq n_0+r,\vspace{1mm}\\
d^0_i : \mathbf {\Lambda} _{n_0,r,k}X \rightarrow \mathbf {\Lambda}_{i-1,n_0+r-i,k+1}X & \text{for}\ \  1\leq i\leq n_0
,\vspace{1mm}\\
d^0_i : \mathbf {\Lambda} _{n_0,r,k}X \rightarrow \mathbf {\Lambda}_{n_0,r-1,k+1}X&  \text{for}\ \  n_0< i\leq n_0+r,\vspace{1mm}\\
d^2_i : \mathbf {\Lambda} _{n_0,r,k}X \rightarrow \mathbf {\Lambda}_{n_0-i,r+i-1,k+1}X & \text{for}\ \   1\leq i\leq n_0,
\end{array}
  \]
and \textit{closed necklical degeneracy maps}
\[ \begin{array}{ll} \eta_j : \mathbf {\Lambda} _{n_0,r,k}X \rightarrow \mathbf {\Lambda}_{n_0+1,r,k}X & \text{for}\ \  1\leq j\leq n_0, \vspace{1mm}\\
\eta_j : \mathbf {\Lambda} _{n_0,r,k}X \rightarrow \mathbf {\Lambda}_{n_0,r+1,k}X & \text{for}\ \  n_0<j\leq n_0+r+k+1
\end{array}
 \]
 by
 $d^{\epsilon}_i[f: R \to X] = [f \circ \delta_{\epsilon}^i: R^0_{\epsilon} \to X ]$, where $\epsilon=0,1,2$ and we have written $R^i_{\epsilon}$ for the domain of the injective morphism of closed necklaces $\delta_{\epsilon}^i:R^i_{\epsilon} \to R$ as described in Remark 1, and
$ \eta_j[f: R \to X]=[ f \circ s^j: R^j \to X]$ where $s^j: R^j \to R$ is the $j$-th map of type $(ii')$ or the $(j-n_0)$-th map of type $(ii)$. These maps satisfy certain relations as described in  \cite{saneFREE}.
To define the geometric realization of $\mathbf{\Lambda}X$ we must fix the modeling polytopes as in the following subsection.

\subsection{The  polytopes $F_n$}

We describe the polytopes $F_n$ in two different, but combinatorially equivalent, ways. We begin with a geometric description. Denote by $|\Delta^n|$ and $|I^n|$ the topological standard $n$-simplex and $n$-cube, respectively (we reserve the notation $\Delta^n$ and $I^n$ for abstract sets as in \cite{RS}), and let $v_0,...,v_n$ be the vertices of $|\Delta^n|$. First, consider the following procedure to obtain a polytope combinatorially equivalent to $|I^n|$ from a sequence of truncations applied to $|\Delta^n|$ for any $n\geq 1$. We may think of $|I^1|$ as obtained from $|\Delta^1|$ by cutting via a line that intersects $|\Delta^1|$ in the ``middle". We obtain the combinatorial polytope $|I^2|$ by intersecting the $2$-simplex $|\Delta^2|$ with a line which has $v_0$ on one side and $v_1$ and $v_2$ on the other side. We truncate $|\Delta^2|$ at this line so that $v_0$ has been cut out and replaced by a $1$-simplex, which we label by $(v^0_0, v^1_0)$. The resulting polytope is combinatorially equivalent to $|I^2|$ with its four $1$-faces given by $(v^0_0, v^1_0), (v^0_0, v_1), (v_1, v_2),$ and $(v_1^0, v_2)$. For any integer $n$ this process may be described inductively: intersect $\sigma:=|\Delta^n|$ with a hyperplane $P_0$ which has $v_0$ on one side and all the other vertices on the other side, then remove everything on the side of $v_0$ to obtain a triangular prism $\sigma'$ in which the vertex $v_0$ has been replaced by an $(n-1)$-simplex $\sigma_0:=(v^0_0,v^1_0...,v^{n-1}_0)$ and the second triangular face of $\sigma$ is the $(n-1)$-simplex $\sigma_1=(v_1, v_2,...,v_n)$. By induction we know how to obtain $(n-1)$-cubes from the $(n-1)$-simplices $\sigma_0$ and $\sigma_1.$ Taking this into account we now truncate $\sigma'$ by a hyperplane $P_1$ such that in one side of $P_1$ we have the vertices $v^0_0$ and $v_1$ and all the other vertices on the second side. We continue this process $(n-1)$  times with hyperplanes $P_0, P_1, ...,P_{n-2}$ to obtain a polytope combinatorially equivalent to $|I^n|$. Note that the resulting cube has vertex set $V_0 \cup ... \cup V_{n-2} \cup \{v_{n-1}\} \cup \{ v_n \}$ where $V_i$ denotes the (naturally ordered) set of vertices obtained from consecutive truncations that started at $v_i$; denote by ${v}^0_i$ the first vertex of $V_i$.

Define $F_0= |I^0|$ and $F_1=|I^1|$. For any $n\geq 2$, $F_n$ is obtained from $|\Delta^n|$ by first applying a sequence of truncations using hyperplanes $P_0,...,P_{n-2}$ as described above to obtain a cube $|I^n|$ and then applying a sequence of truncations to the cube at vertices $v_n, v_{n-1}, {v}^0_{n-2},..., {v}^0_2$ using hyperplanes $Q_0,...,Q_{n-2}$, respectively, in the same way as described above (but note that we are now going backwards with respect to the ordering of the vertices).

We now give a combinatorial description of the polytopes $F_n$ as a subdivision of the standard cube $|I^n|$. The polytopes $F_n$ will have $3n-1$ faces of codimension $1$: $n$ corresponding to the symbols $d^0_1,...,d^0_n$, $n-1$ corresponding to the symbols $d^1_2,...,d^1_n$, and $n$ corresponding to the symbols $d^2_1=d^1_1, d^2_2,...,d^2_n$. We may think of the symbols $d^{\epsilon}_i$ for $\epsilon=0,1,2$, $i=1,...,n$ as labels; we call them \textit{face operators} since they will fit with the closed necklical face maps introduced in section 2; in fact, we have used the same notation.

Let $F_0$ be a point. Let $F_1$ be the interval $[0,1]$ with first vertex corresponding to $d^0_1$ and second vertex corresponding to $d^1_1=d^2_1$. Denote by $e^{\epsilon}_{i}$
the face $(x_1,...,x_i,\epsilon,x_{i+1},....,x_{n-1})$ of the cube $I^n$ where $\epsilon=0,1$
and $1\leq i\leq n.$  Suppose $F_{n-1}$ has been constructed as a subdivision of $|I^{n-1}|$. We define  inductively $F_n$ to be the subdivision of $F_{n-1}\times [0,1]$ with faces given by the following table:

%\vspace{0.1in}

\begin{equation*}
\begin{tabular}{c|cc}
$\underset{\ }{\text{\textbf{Face of }}F_{n}}$ & \textbf{Face operator} &
\vspace{-1mm}\\
 \hline

 $e^{0}_{i}$ & $d^{0}_{i} ,$ & \multicolumn{1}{l}{$1\leq i\leq n,$} \\

$e^{1}_{i}$ & $d^{1}_{i},$ & \multicolumn{1}{l}{$2\leq i\leq n,$}
 \\

$d^{2}_{i}\times [0,1]$ & $d^{2}_{i},$ & \multicolumn{1}{l}{$1\leq i\leq n-2,$} \\

$d^{2}_{n-1}\times\left[0,\frac{1}{2}\right]^{\strut } $ & $d^{2}_{n-1},$
& \multicolumn{1}{l}{} \\

$d^{2}_{n-1}\times\left[\frac{1}{2},1  \right]^{\strut } $ & $d^{2}_{n}.$ &
\multicolumn{1}{l}{}
\end{tabular}
\end{equation*}
\vspace{0.2in}

 Thus, $F_2$ is a pentagon, $F_3$ has eight 2-faces (4 pentagon and 4 quadrilateral), 18
edges and 12 vertices, as showed in Figure 1.

\vspace{0.2in}

% This is a LaTeX picture output by TeXCAD.
% File name: [F2-F3.pic].
% Version of TeXCAD: 4.3
% Reference / build: 30-Jun-2012 (rev. 105)
% For new versions, check: http://texcad.sf.net/
% Options on the following lines.
%\grade{\on}
%\emlines{\off}
%\epic{\off}
%\beziermacro{\on}
%\reduce{\on}
%\snapping{\off}
%\pvinsert{% Your \input, \def, etc. here}
%\quality{8.000}
%\graddiff{0.005}
%\snapasp{1}
%\zoom{4.7568}
\unitlength 1mm % = 2.845pt
\linethickness{0.4pt}
\ifx\plotpoint\undefined\newsavebox{\plotpoint}\fi % GNUPLOT compatibility
\begin{picture}(109.471,38.739)(0,0)
\put(16.998,4.334){\line(1,0){25.375}}
\put(15.428,35.657){\line(1,0){25.375}}
\put(42.373,4.334){\line(0,1){21}}
\put(42.373,25.334){\line(-1,0){25.375}}
\put(16.998,25.334){\line(0,-1){20.875}}
\put(42.373,25.209){\circle*{1}}
\put(17.123,25.084){\circle*{1}}
\put(16.998,4.334){\circle*{1}}
\put(15.428,35.657){\circle*{1}}
\put(42.373,4.334){\circle*{1}}
\put(40.803,35.657){\circle*{1}}
\put(42.373,14.334){\circle*{1}}
\put(64.846,35.722){\line(1,0){22.875}}
\put(64.846,35.472){\circle*{1}}
\put(85.346,8.597){\circle*{1}}
\put(108.971,8.472){\circle*{1}}
\put(98.971,12.472){\circle*{1}}
\put(87.971,16.722){\circle*{1}}
\put(87.971,26.597){\circle*{1}}
\put(99.096,21.972){\circle*{1}}
\put(98.721,31.597){\circle*{1}}
\put(108.846,27.722){\circle*{1}}
\put(85.596,27.597){\circle*{1}}
\put(87.846,35.597){\circle*{1}}
%\emline(87.971,35.472)(109.221,27.722)
\multiput(87.971,35.472)(.092391304,-.033695652){230}{\line(1,0){.092391304}}
%\end
%\emline(64.846,35.347)(85.471,27.597)
\multiput(64.846,35.347)(.089673913,-.033695652){230}{\line(1,0){.089673913}}
%\end
\put(85.221,27.597){\line(1,0){23.875}}
\put(85.471,27.347){\line(0,-1){18.875}}
\put(85.471,8.472){\line(1,0){23.5}}
\put(108.971,8.472){\line(0,1){19.125}}
\put(87.971,35.597){\line(0,-1){7.625}}
\put(98.971,31.347){\line(0,-1){3.5}}
%\emline(88.096,26.472)(99.096,21.972)
\multiput(88.096,26.472)(.082089552,-.03358209){134}{\line(1,0){.082089552}}
%\end
\put(88.096,26.972){\line(0,-1){10.25}}
%\emline(88.096,16.722)(108.971,8.597)
\multiput(88.096,16.722)(.086618257,-.033713693){241}{\line(1,0){.086618257}}
%\end
\put(85.846,16.597){\line(1,0){2.375}}
\put(84.971,16.722){\line(-1,0){19.75}}
%\emline(65.221,16.722)(85.471,8.722)
\multiput(65.221,16.722)(.085084034,-.033613445){238}{\line(1,0){.085084034}}
%\end
\put(64.971,35.347){\line(0,-1){18.5}}
\put(93.771,30.347){\makebox(0,0)[cc]{$_{3][0123]}$}}
\put(93.771,20.222){\makebox(0,0)[cc]{$_{23][012]}$}}
\put(104.096,19.472){\makebox(0,0)[cc]{$_{123][01]}$}}
\put(99.096,26.972){\line(0,-1){14.5}}
\put(65.096,16.847){\circle*{1}}
\put(83.346,32.097){\makebox(0,0)[cc]{$_{013]}$}}
\put(77.846,25.847){\makebox(0,0)[cc]{$_{023]}$}}
\put(74.221,19.722){\makebox(0,0)[cc]{$_{0][0123]}$}}
\put(79.221,14.472){\makebox(0,0)[cc]{$_{012][23]}$}}
\put(92.221,11.847){\makebox(0,0)[cc]{$_{01][123]}$}}
\put(16.728,1.666){\makebox(0,0)[cc]{$_{0][01][12]}$}}
\put(42.353,2.291){\makebox(0,0)[cc]{$_{1][12][01]}$}}
\put(49.228,14.041){\makebox(0,0)[cc]{$_{2][01][12]}$}}

\put(47.5,24.512){\makebox(0,0)[cc]{$_{2][02]}$}}
\put(12.5,24.512){\makebox(0,0)[cc]{$_{0][02]}$}}
\put(21.603,14.334){\makebox(0,0)[cc]{$_{0][012]}$}}
\put(37.998,18.334){\makebox(0,0)[cc]{$_{2][012]}$}}
\put(29.623,23.209){\makebox(0,0)[cc]{$_{02]}$}}
\put(29.248,6.116){\makebox(0,0)[cc]{$_{01][12]}$}}
\put(38.103,9.209){\makebox(0,0)[cc]{$_{12][01]}$}}

\put(15.743,33.0){\makebox(0,0)[cc]{$_{0][01]}$}}
\put(41.022,33.0){\makebox(0,0)[cc]{$_{1][01]}$}}
\put(27.449,38.739){\makebox(0,0)[cc]{$01]$}}
\put(8.0,13.665){\makebox(0,0)[cc]{$012]$}}
\put(118.0,18.92){\makebox(0,0)[cc]{$0123]$}}
\end{picture}

%\vspace{-0.2in}
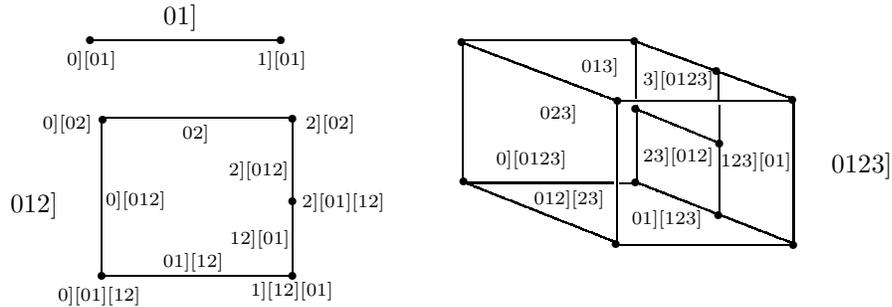
\begin{center} Figure 1: $F_{n}$ as a subdivision of $F_{n-1}\times [0,1]$
for $n=1,2,3.$
\end{center}

\vspace{0.2in}

We will now relate the polytopes $F_n$ to closed necklical sets and describe a combinatorial labeling for the cells of $F_n$ (which have already appeared in Figure 1) similar to the labeling described in section 3.2 of \cite{RS} for the cells of a cube. Let $\mathbb{F}_n:= Y(\Delta^n \vee \Delta^1)$ where $Y: Nec_c \to Set_{Nec_c}$ is the Yoneda embedding.
The \textit{codimension j faces} of the closed necklical set $\mathbb{F}_n$ are defined to be those elements of $\mathbb{F}_n$ given by injective morphisms of closed necklaces $R \to \Delta^n \vee \Delta^1$ where $R$ is a closed necklace of dimension $n-j$. It follows that the codimension 1 faces of $\mathbb{F}_n$ are given by morphisms $R \to \Delta^n \vee \Delta^1$ of closed necklaces where $R$ is  $\Delta^{n-1} \vee \Delta^1$, $\Delta^p \vee \Delta^q \vee \Delta^1$, or $\Delta^{p} \vee \Delta^1 \vee \Delta^q$ for some integers $p\geq 0, q\geq 1$ satisfying $p+q=n.$  The only possible  injective maps of closed necklaces for such $R$ are given by the following ``co-face" morphisms as described in Remark 1

\[
\begin{array}{ll}
\delta^i_0: \Delta^{i-1}\vee \Delta^{n-i+1} \vee \Delta^1 \hookrightarrow \Delta^n \vee \Delta^1        & \text{for}\ \ 1\leq i\leq n,\\
\delta^i_2: \Delta^{n-i}\vee \Delta^1 \vee \Delta^i \hookrightarrow \Delta^n \vee \Delta^1        & \text{for}\ \ 1\leq i\leq n,\\
\delta^i_1=d^{i-1} \vee id_{_{\Delta^1}}: \Delta^{n-1} \vee \Delta^1 \hookrightarrow \Delta^n \vee \Delta^1          & \text{for}\ \ 2\leq i\leq n,
\end{array}
    \]
and, by convention, set $\delta^1_1:=\delta^1_2.$

 Labeling the top cell of $\mathbb{F}_n$ (i.e. the element $id: \Delta^n \vee \Delta^1 \to \Delta^n \vee \Delta^1$ of $\mathbb{F}_n$) by the symbol $0,1,...,n],$ the above morphisms $\delta^i_0, \delta^i_2$,  $\delta^i_1$ determine codimension 1 faces of $\mathbb{F}_n$ which may be labeled by the symbols
 $0,1,...,i][i,i+1,...,n]$, $i,i+1,...,n][0,...,i]$, and $0,...,i-1,\hat{i},i+1, ....,n]$, respectively, where $\hat{i}$ means we omit $i$ from such sequence. In general, any proper face $u$ of $0,...,n]$ may be labeled by a symbol

\begin{multline}
\!\!\!\!\!u\!= i_{s_t},\!...,i_{s_{t+1}}][i_{s_{t+1}},\!...,i_{s_{t+2}}]
...[i_{s_{k-1}},\!...,i_{s_{k}},\!n] [0,i_1,\!...,i_{s_1}][i_{s_1},\!...,i_{s_{2}}]
...[i_{s_{t-1}},\!...,i_{s_t}],
\\
0<i_{1}<\ldots <i_{s_{t}}<\ldots <i_{s_{k}}<n,
\end{multline}
 where $\dim u=s_k-k+1.$

The combinatorics of the faces of the closed necklical set $\mathbb{F}_n$ agree with those of the topological space $F_n$, meaning that the above symbols may be used to label the cells of the polytope $F_n$ as well (see Figure 1). Moreover, as explained in \cite{RS}, the top dimensional cell of the standard cube $|I^n|$ may be labeled by $[0,1...,n+1]$ and its $2n$ codimension $1$ faces by  $[0,1...,i][i,i+1,...,n+1]$ and $[0,...,i-1,\hat{i},i+1,...,n+1]$ for $i=1,...,n$ with face operators also denoted by $d^0_i$ and $d^1_i$, respectively.

 It follows that we may think of $a_0,a_1,\!...,a_p][b_0,\!...,b_{q+1}]$ as a label for the top cell of $F_p \times |I^q|$ together with face operators $d^0,d^1,d^2$ given by
\[
\begin{array}{llll}
 a_0,a_1,\!...,a_p][b_0,\!...,b_{q+1}]\overset{d^{0}_{i}}{\longrightarrow}\!
\left\{\!\!\!
 \begin{array}{lll}
a_0,a_1,\!...,a_{i-1}][a_{i-1},\!...,a_{p}][b_0,\!...,b_{q+1}], & 1\!\leq\! i\!\leq\! p,\\
a_0,a_1,\!...,a_{p}][b_0,...,b_{j}][b_j,...,b_{q+1}],& i\!=p\!+\!j\\
\end{array}
\right. \\
\\
a_0,a_1,\!...,a_m][b_0,\!...,b_{n+1}]\overset{d^{1}_{i}}{\longrightarrow}\!
 \left\{\!\!\! \!  \begin{array}{llll}
a_1,\!...,a_p][b_0,\!...,b_{q+1}][a_0,a_1], \  & && i=1,\\
a_0,a_1,\!...,\hat{a}_{i-1},\!...,a_p][b_0,\!...,b_{q+1}], \  & & &  2\!\leq \!i\!\leq\! p,\\
a_0,a_1,\!...,a_{p}][b_0,\!b_1,\!...,\hat{b}_{j},\!...,b_{q+1}], \  &&  &  i\!=p\!+\!j\\
\end{array}
\right.
 \\
a_0,a_1,\!...,a_p][b_0,\!...,b_{q+1}]\!\overset{d^{2}_{i}}{\longrightarrow}
a_i,\!...,a_{p}][b_0,\!...,b_{q+1}][a_0,a_1,\!...,a_i], \hspace{0.5in}  1\!\leq i\leq p.
\end{array}
\]
Thus, the proper faces  of $F_n$ obtained by means of  the operators $d^0$ and $d^2$ can be identified with the top cell of cartesian products $F_p\times I^q$
for $0\leq p,q\leq n$ with $p+q<n.$
These identifications induce continuous inclusion maps which we denote by the same symbols
\[
\begin{array}{ll}
\delta^i_0: F_{i-1}\times |I^{n-i}|\hookrightarrow F_n               & \text{for}\ \ 1\leq i\leq n,\\
\delta^i_2: F_{n-i}\times |I^{i-1}|\hookrightarrow F_n                & \text{for}\ \ 1\leq i\leq n,     \\
\delta^i_1: F_{n-1} \hookrightarrow F_n               & \text{for}\ \ 2\leq i\leq n.
\end{array}
    \]

\vspace{0.2in}

In a similar way, the morphisms of type $(ii)$ and $(ii')$ induce continuous collapse maps $\varsigma^j: F_{p+1} \times |I^q| \to F_p \times | I^q|$ for $j=1,...,p+1$ and $\varsigma^j: F_p \times |I^{q+1}| \to F_p \times |I^q|$ for $j=p+2,..., p+q+1$, respectively.  We denote by $\partial^i_{\epsilon}: |I^{n-1}| \to |I^n|$ the usual cubical face inclusion maps.

There is a canonical continuous cellular projection map
\begin{equation}\label{varphi}
\varphi: F_n\rightarrow |\Delta^n|
\end{equation}
which sends the top cell $0,...,n]$ of $F_n$ to the top cell $(0,...,n)$ of the simplex $|\Delta^n|$ in such a way that a face $u$  of $F_n$ as in $(3.3)$ goes to the face $(i_{s_t},...,i_{s_{t+1}})$ of $(0,...,n)=|\Delta^n|$. In particular, $\varphi$ collapses those  faces  of $F_n$ with labels  $$i_{s_t}][i_{s_{t}},\!...,i_{s_{t+1}}]
...[i_{s_{k-1}},\!...,i_{s_{k}},\!n] [0,i_1,\!...,i_{s_1}][i_{s_1},\!...,i_{s_{2}}]
...[i_{s_{t-1}},\!...,i_{s_t}]$$ to the vertex $i_{s_t}$ of $|\Delta^n|$.

\subsection{The geometric realization of $\mathbf{\Lambda} X$}

The \textit{geometric realization} of the closed necklical set $\mathbf{\Lambda} X $  is
defined analogously to that of $\mathbf{P}X$ and $\mathbf{\Omega}X$ as in \cite{RS}, but  instead of the standard cubes as modeling polytopes we use the cartesian products $F_{n_0}\times I^r.$
Namely,  $ |\mathbf{\Lambda} X|$
is
the topological space defined by
\begin{eqnarray*}
|\mathbf{\Lambda} X|:= \bigsqcup_{n,r\geq 0}\mathbf{\Lambda}_{n_0,r} X  \times (F_{n_0}\times  |I^{r}|)   / \sim,
\end{eqnarray*}
where  $\mathbf{\Lambda}_{n_0,r} X $ is considered as a topological space with the discrete topology, and $\sim$ is the equivalence relation  generated as follows:
for any $f: \Delta^{n_0} \vee T \to X$
$ (f, \delta^i_{\epsilon} (t), s )\sim (d^{\epsilon}_i(f), t,s) $ for $\epsilon=0,1,2,$
  $ (f, t' , \partial^i_{\epsilon} (s') )\sim (d^{\epsilon}_i(f), t',s') $ for $\epsilon=0,1$
   and $(f, \varsigma^j(t))\sim (\eta_j(f), t).$ Here we have written $(t,s)$ to denote points in $(F_{i-1}\times|I^{n_0-i}|) \times |I^r|$ when $\epsilon=0$, points in $(F_{n_0-i} \times |I^{i-1}|) \times |I^r|$ when $\epsilon=2$, points in $F_{n_0-1} \times |I^r|$ when $\epsilon=1$, and $(t',s')$ to denote points in $F_{n_0} \times |I^{r-1}|$. In particular, we have $|\mathbb{F}_n|= F_n$.  Compare $|\mathbf{\Lambda}(X)|$ to $|\mathbf{\Omega} (X;x)|$ which was defined in \cite{RS} as the space
\begin{eqnarray*}
|\mathbf{\Omega} (X;x)|:= \bigsqcup_{n\geq 0}\mathbf{\Omega}_n (X;x)  \times | I^n |  / \sim
\end{eqnarray*}
where $\sim$ is defined similarly to the relation above.

\subsection{Inverting $1$-simplices formally}

Given a simplicial set $X,$ form a set $X_1^{op}:=\{ x^{op} \mid x \in X_1\ \ \text{is non-degenerate} \}.$ Let $Z(X)$ be the minimal simplicial set containing the set
$X\cup X^{op}_1$ such that $\partial_0( x^{op} )= \partial_1( x)$ and  $\partial_1( x^{op} ) = \partial_0 (x).$
Define
\begin{eqnarray*}
\widehat{\mathbf{\Lambda}}: Set_{\Delta} \to Set_{Nec_c}
\end{eqnarray*}
as $\widehat{\mathbf{\Lambda}}(X):= \mathbf{\Lambda}(Z(X)) / \sim$
where the equivalence relation $\sim$ is generated by
\[ f\sim  f'\circ g:T\rightarrow Z (X) \]
for all $T=\Delta^{n_0} \vee \Delta^{n_{1}}\vee ...\vee\Delta^{n_p}\vee\Delta^{n_{p+1}}\vee ...\vee \Delta^{n_{k}} $ with $1 \leq p \leq k$, $n_p=n_{p+1}=1$, and
morphisms $f$ and $g$ such that  $f$ satisfies $f(\Delta^{n_p})=(f(\Delta^{n_{p+1}}))^{op},$ so $f$ induces a map
$f':  \Delta^{n_0} \vee \Delta^{n_{1}}\vee ...\vee \Delta^{n_{p-1}}\vee\Delta^{n_{p+2}}\vee ...\vee \Delta^{n_{k}}\rightarrow X,$
and
\begin{multline*}
g:  \Delta^{n_0} \vee \Delta^{n_{1}}\vee ...\vee \Delta^{n_{p-1}}\vee\Delta^{n_p}\vee\Delta^{n_{p+1}}\vee\Delta^{n_{p+2}}\vee ...\vee \Delta^{n_{k}}
\rightarrow
\\
 \Delta^{n_0} \vee \Delta^{n_{1}}\vee ...\vee \Delta^{n_{p-1}}\vee\Delta^{n_{p+2}}\vee ...\vee \Delta^{n_{k}}
\end{multline*}
is the collapse map. The geometric realization $|\widehat{\mathbf{\Lambda}}X|$ is defined similarly as in the previous section, but note that now there are new $0$-cells corresponding to the inverted $1$-simplices.

\subsection{An explicit construction of $\widehat{\mathbf{\Lambda}}X$}

Let $(X,x_0)$ be a pointed simplicial set with face and degeneracy maps denoted by $\partial_i$ and $s_j$, respectively. For a simplex $\sigma \in X$ denote by $\min \sigma$ and $\max \sigma$ the first and last vertices of $\sigma$, respectively.
 We recall the following explicit description of the underlying graded set $\{ \widehat{\mathbf{\Omega}}_n X \}_{n\geq 0}$ of the necklical set $\widehat{\mathbf{\Omega}} X$. For any $\sigma_i\in Z(X)_{>0}$, let $\dim (\bar \sigma)=\dim(\sigma)-1 $ and define
 \begin{multline}
\widehat{\mathbf{\Omega}}'_n X=\{ \bar \sigma_1\cdots \bar \sigma_k  \mid \max \sigma_i=\min \sigma_{i+1}\
\text{for all}\ i,\ \max\sigma_k= x_0 ,\\  \, | \bar \sigma_1|+\cdots +| \bar \sigma_k|=n,\, k\geq 1  \}
  \end{multline}
with relations
\[\bar \sigma_1\cdots  \bar \sigma_i\cdot \bar \sigma_{i+1}   \cdots \bar \sigma_k=
\bar \sigma_1\cdots \bar \sigma_{i-1}\cdot \bar \sigma_{i+2}\cdots    \bar \sigma_k
\ \ \text{and}\ \  \bar \sigma_i\cdot \bar \sigma_{i+1}= \overline{s_0( x)}
 \]
where $\sigma_{i},\sigma_{i+1}\in Z(X)_1$ such that $\sigma_{i+1}=\sigma^{op}_i$ and $x=\min \sigma_i$; and
\[\bar \sigma_1\cdots  \overline{s_{n_i}( \sigma_i)}\cdot \bar \sigma_{i+1}\cdots \bar \sigma_k=
\bar \sigma_1\cdots  \bar \sigma_{i}\cdot \overline{s_0( \sigma_{i+1})} \cdots \bar \sigma_k\ \ \text{for}   \ \ i=1,...,k-1.\]
Then
\[\widehat{\mathbf{\Omega}}_n X \cong \{ \bar \sigma_1\cdots \bar \sigma_k \in
\widehat{\mathbf{\Omega}}'_n X \mid  \min \sigma_{1}=x_0\},\]
and the monoidal structure $ \widehat{\mathbf{\Omega}}X \times  \widehat{\mathbf{\Omega}}X \to  \widehat{\mathbf{\Omega}}X$ is induced by concatenation of words with unit $e=s_0( x_0).$
In particular,     $ \widehat{\mathbf{\Omega}}_0 X$ is a group. Let $\widehat{\mathbf{\Omega}}''X$ be the set  obtained from  $\widehat{\mathbf{\Omega}}'X$ by removing the condition
 $\max \sigma_k=x_0.$

Let  $\widehat{\mathbf{\Lambda}} X=\{\widehat{\mathbf{\Lambda}}_n X\}_{n\geq 0}$ be  a set
 $\widehat{\mathbf{\Lambda}} X={\mathbf{\Lambda}}' X/\sim, $
where
${\mathbf{\Lambda}}' X$ is   a subset of the (set-theoretical) cartesian product
   $X\times \widehat{\mathbf{\Omega}}''X$  of two graded sets $X$ and $\widehat{\mathbf{\Omega}}''X:$
\begin{equation}
{\mathbf{\Lambda}}'_n X \! =\!
\{\! (x\,,\, \bar \sigma_1\cdots \bar \sigma_k ) \in\bigcup_{p+q=n}\! X_{p}\!\times \widehat{\mathbf{\Omega}}''_{q} X \mid
\min x=\max \sigma_k,\, \max x=\min \sigma_1\}\!
\end{equation}
and $\sim$ is defined
via the relation
\[ (s_{p}(x),y)\sim (x,\eta_1(y)). \]
The face operators
 \[d^0_i,d^1_i,d^2_i: \widehat{\mathbf{\Lambda}}_n X \rightarrow
 \widehat{\mathbf{\Lambda}}_{n-1} X \]
  are  defined
  for
 $ (x\,,  y)\in X_p\times \widehat{\mathbf{\Omega}}''_q X
\rightarrow   \widehat{\mathbf{\Lambda}}_n X $ ( $p+q=n$)  by
\[
 \begin{array}{llllll}
 d^0_i(x,y)= \left\{
   \begin{array}{llll}
  \left (\min x\, ,\,    \bar x  \cdot y \right), &  & i=1 ,
   \newline \vspace{1mm}\\
   \left( {\partial _{i}\cdots \partial _{p+1}(x)}\, ,\, \overline {\partial _0\cdots \partial _{i-2}(x)}    \cdot y \right), & &  2\leq i \leq p,           \newline \vspace{1mm}  \\
  (x\, , d^0_{i-p}(y))   , & &  p<i\leq n,
   \end{array}
 \right.
 \\
 \\
 d^1_i(x,y) = \left\{
   \begin{array}{llllllllll}
   ( \partial _{i-1}(x)\, ,  y), &  \hspace{1.50in}  1\leq i \leq p,\newline \vspace{1mm}\\
   (x ,d^1_{i-p}(y))   , &        \hspace{1.50in}   p<i\leq n,
   \end{array}
 \right.
 \\
 \\
 d^2_i(x,y)=\,\,\left(\, \partial_{0}\cdots \partial_{i-1}(x)\,,\, y\cdot \overline{\partial_{i+1}\cdots\partial_{m}(x)}\,\right),
 \hspace{0.35in}  1\leq i \leq p,
 \end{array}
\]
and
the degeneracy maps
\[  \eta_j :  \widehat{\mathbf{\Lambda}} _{n,k}(X) \rightarrow
 \widehat{\mathbf{\Lambda}}_{n+1,k}(X)\ \ \text{for}\ \  j=1,...,n+k+1  \]
by
\begin{equation*}
\eta_j(x,y)=\left\{ \begin{array}{llll} (s_{j-1}(x), y),  &1\leq j\leq p+1, \\
                       (x, \eta_{j-p}(y)), & p+1\leq j\leq n+1.
\end{array}
\right.
\end{equation*}

We have  the  short sequence
\[
 \widehat{\mathbf{\Omega}} X   \overset{i }{\longrightarrow}
  \widehat{\mathbf{\Lambda}}X   \overset{pr}{\longrightarrow} X
\]
of maps of sets where $i$ is defined by  $i(y)=(x_0, y)$, for any $y \in \widehat{\mathbf{\Omega}} X $
while $pr(x,y)=x$ for
$(x,y)\in   \widehat{\mathbf{\Lambda}} X.$
The set map $i: \widehat{\mathbf{\Omega}} X \to \widehat{\mathbf{\Lambda}}X$ induces a continuous map
 $ \iota: |\widehat{\mathbf{\Omega}} X|  \to |\widehat{\mathbf{\Lambda}}X |$. The projection
$pr : \widehat{\mathbf{\Lambda}}X  {\longrightarrow} X$ together with the continuous map $ \varphi:F_n\rightarrow \Delta^n$
induces a continuous and cellular map $\zeta: |\widehat{\mathbf{\Lambda}}X|  \to |X|$.

\begin{proposition}\label{quasifree}

For a pointed connected simplicial set  $(X,x_0)$  the short sequence
\[
 |\widehat{\mathbf{\Omega}} X|   \overset{\iota}{\longrightarrow}
  |\widehat{\mathbf{\Lambda}}X|   \overset{\zeta}{\longrightarrow} |X|
\]
is  a quasi-fibration.
\end{proposition}
\begin{proof}
The result follows from essentially the same argument as in Proposition 2(iii) in \cite{RS} which we recall here for completeness.
Recall $|X|$ is a space defined as a colimit of standard topological simplices with identifications given by the face and degeneracy maps of $X$. Take the barycentric subdivision of each standard simplex in the colimit to obtain a finer subdivision of $|X|$ into simplices. For each simplex $\sigma \subset |X|$ in this subdivision let $U_{\sigma}$ be a small open neighborhood containing $\sigma$ as a deformation retract; in particular, each $U_{\sigma}$ is contractible.  Let $\mathcal{U}$ be the smallest collection of open sets containing $\{U_{\sigma} \}$ which is closed under finite intersections. Then $\mathcal{U}$ is an open covering of $|X|$ with the property that for any $U \in \mathcal{U}$ and any $x \in U$, $\zeta^{-1}(x)\hookrightarrow \zeta^{-1}(U)$ is a homotopy equivalence. It follows that $\zeta$ satisfies the criterion in \cite{Dold- Thom}  to be a quasi-fibration.
\end{proof}

A necklical model for the path fibration was  described in \cite{RS}. More precisely, the following theorem was proved.

\begin{theorem}\label{loopmodel}
Let $Y=|X|$ be the geometric realization of a path connected simplicial set $X.$
Let $ \Omega Y \overset{i}{\rightarrow} PY\overset{\pi} \rightarrow Y$ be  the path fibration on $Y.$ Then  there is a  commutative diagram
\begin{equation}\label{diagram}
\begin{array}{cccccc}
   |\widehat{\mathbf{\Omega}} X| & \overset{\omega}{\longrightarrow}                  &                 \Omega Y \\
        \iota  \downarrow            &                  &      \hspace{-0.1in} \iota   \downarrow \\
      |\widehat{\mathbf{P}}   X|  &  \overset{p}{\longrightarrow} &                     PY                    \\
     \xi \downarrow    &                                    &       \hspace{-0.1in}\pi     \downarrow           \\
 \hspace{0.1in}|X| & \overset{Id}{\longrightarrow}                    &                     Y
\end{array}
\end{equation}
in which
  $\omega$ is a monoidal map and homotopy equivalence.
\end{theorem}
In the above statement   $|\widehat{\mathbf{P}}   X|$ is a contractible cellular space obtained by glueing cubes corresponding to \textit{augmented necklaces} inside $X$ with last vertex fixed at $x_0 \in X_0$ and $PY$ denotes the space of paths with endpoint fixed at $b$ and $\Omega Y$ the space of loops based at $b$. An augmented necklace is just a necklace in which the first bead is allowed to be $\Delta^0$; see \cite{RS} for more details. In the proof of our the next theorem we recall the construction of $\omega$.

The main result of this article is the following analogue result for the free loop fibration.

\begin{theorem}\label{freeloopmodel}
Let $Y=|X|$ be the geometric realization of a path connected simplicial set $X.$
Let $ \Omega Y \overset{i}{\rightarrow} \Lambda Y\overset{\varrho} \longrightarrow Y$ be  the free loop fibration on $Y.$ There is a  commutative diagram
\begin{equation}\label{freediagram}
\begin{array}{cccccc}
   |\widehat{\mathbf{\Omega}} X| & \overset{\omega}{\longrightarrow}                  &                 \Omega Y \\
        \iota  \downarrow            &                  &      \hspace{-0.1in} \iota   \downarrow \\
      |\widehat{\mathbf{\Lambda}}   X|  &  \overset{\Upsilon}{\longrightarrow} &                     \Lambda Y                    \\
     \zeta \downarrow    &                                    &       \hspace{-0.1in}  \varrho     \downarrow           \\
 \hspace{0.1in}|X| & \overset{Id}{\longrightarrow}                    &                     Y
\end{array}
\end{equation}
in which
  $\Upsilon$ and $\omega$ are homotopy equivalences.
\end{theorem}
\begin{proof}
First consider the cellular map $\varphi$  given by (\ref{varphi}). We need to represent the underlying cube $I^n=I^{n-1}\times I$ of $F_n$ so that the both $(n-1)$-cubes   $I^{n-1}\times 0$ and $I^{n-1}\times 1$ (labelled by $d^0_1(F_n)$  and $d^2_i(F_n)$) are totally degenerated under $\varphi$ at the vertices $v_0 $ and $v_n$ respectively. For example,
  only subdivide   the  faces  $d^0_i(I^n)$  for $2\leq i\leq n$  to obtain the two cells $d^0_i(F_n)$ and $d^2_i(F_n)$   for each $i$ such that the  vertices of $F_n$ lying in
$d^0_i(F_n)\cap d^2_i(F_n)\cap e $  for $i=2,...,n$  form an increasing sequence on the edge $e:=(x_1,0,...,0)$ of $I^n$ (see Figure 2).
Then fix the map
\[\omega_{n}: |I^{n-1}|\rightarrow  P_{0,n}|\Delta^n|  \subset |\Delta^n|^{I}  \]
by means of the exponential law applied for $|\varphi|: |I^{n-1}|\times |I|\rightarrow |\Delta^n| $, where $P_{0,n}|\Delta^n|$ denotes the space of paths in $|\Delta^n|$ starting at $v_0$ and ending at $v_n$.

Recall the space $|\widehat{\mathbf{\Omega}} X|$ is constructed by glueing cubical cells, so any point in $|\widehat{\mathbf{\Omega}} X|$ is given by an equivalence class $y=[f,(\mathbf{t}_1,...,\mathbf{t}_{k})]$ where  $f: T \to X$ is a (non-degenerate) map of simplicial sets
  sending the first and last vertices of a necklace $T$ to a fixed point $x \in X_0,$ and
  $(\mathbf{t}_1,...,\mathbf{t}_{k} ) \in |I^{n_1}|\times \cdots \times |I^{n_k}|.$  Let $|f|: |T | \to Y=|X|$ be the induced map at the level of geometric realizations. Then define
\[\omega(y):= |f| \circ   \omega_T \ \  \text{for}  \  \  \omega_T:= \omega_{n_k}(\mathbf{t}_k)* \cdots * \omega_{n_1}(\mathbf{t}_1) \]
where $*$ denotes path concatenation in $|T|.$

We now construct $\Upsilon$. Any point
 $(x,y)\in
| X_p\times \widehat{\mathbf{\Omega}}'_r X | \to |\widehat{\mathbf{\Lambda}}X |$ is given by an equivalence class $(x,y)=[g,(\mathbf{x}, \mathbf{y})]$ where $g: R \to X$ a (non-degenerate) map of simplicial sets, $R=\Delta^{n_0} \vee T$ a closed necklace, and $(\mathbf{x}, \mathbf{y}) \in F_{n_0} \times |I^{n_1}|\times \cdots \times |I^{n_k}|.$   Let $\lambda:=|g|\circ \omega_T;$ so $\lambda$ is a path from $|g|( |\min T|) $ to $|g|(| \max T | ).$

For $n_0=0,  \Upsilon (x,y)$ is defined as the loop $\omega(y)$ based at $x= |g|( |\Delta^{n_0}| ) \in Y$.  Let $n_0 \geq 1$ and $(\mathbf{t},s)\in F_{n_0}.$  Define $\Upsilon(x,y)=\beta\ast \lambda \ast \alpha ,$ where $\alpha$ is  a path from $|g|\circ |\varphi|(\mathbf{t},s)$ to $|g|\circ |\varphi|(\mathbf{t},1)$ defined by the restriction of $|g|\circ \omega_{n_0}(\mathbf{t})$ to $[s,1]$ and
 $\beta$ is  a path from $|g|\circ |\varphi|(\mathbf{t},0)$ to $|g|\circ |\varphi|(\mathbf{t},s)$ defined by the restriction of $|g|\circ \omega_{n_0}(\mathbf{t})$ to $[0,s].$

Finally, apply Proposition \ref{quasifree}, the fact that a quasi-fibration gives rise to a long exact sequence in homotopy groups, and the result of \cite{RS} that says that $\omega$ is a homotopy equivalence to obtain the desired result.

\end{proof}
\begin{remark} \normalfont
The maps $\Upsilon$ and $\omega$ above are canonically defined by means of the maps $\varphi$ given by (\ref{varphi}). Furthermore,
$\Upsilon$ is an inclusion so that it detects a deformation retract of $\Lambda Y$ having the cellular structure determined by the simplicial structure of $X.$
\end{remark}

% This is a LaTeX picture output by TeXCAD.
% File name: [FF23.pic].
% Version of TeXCAD: 4.3
% Reference / build: 30-Jun-2012 (rev. 105)
% For new versions, check: http://texcad.sf.net/
% Options on the following lines.
%\grade{\on}
%\emlines{\off}
%\epic{\off}
%\beziermacro{\on}
%\reduce{\on}
%\snapping{\off}
%\pvinsert{% Your \input, \def, etc. here}
%\quality{8.000}
%\graddiff{0.005}
%\snapasp{1}
%\zoom{4.0000}
\unitlength 1mm % = 2.845pt
\linethickness{0.4pt}
\ifx\plotpoint\undefined\newsavebox{\plotpoint}\fi % GNUPLOT compatibility
\begin{picture}(107.5,93.862)(0,0)
\put(17.77,44.918){\line(1,0){25.375}}
\put(43.145,44.918){\line(0,1){21}}
\put(43.145,65.918){\line(-1,0){25.375}}
\put(17.77,65.918){\line(0,-1){20.875}}
\put(43.145,65.793){\circle*{1}}
\put(17.895,65.668){\circle*{1}}
\put(19.395,91.216){\circle*{1}}
\put(67.478,91.393){\circle*{1}}
\put(17.77,44.918){\circle*{1}}
\put(43.145,44.918){\circle*{1}}
\put(28.416,31.064){\line(1,0){22.875}}
\put(28.416,30.814){\circle*{1}}
\put(12.666,21.939){\circle*{1}}
\put(12.541,4.439){\circle*{1}}
\put(37.041,4.314){\circle*{1}}
\put(37.041,21.564){\circle*{1}}
\put(20.916,21.564){\circle*{1}}
\put(28.291,13.314){\circle*{1}}
\put(51.291,13.314){\circle*{1}}
\put(28.541,4.314){\circle*{1}}
\put(51.416,30.939){\circle*{1}}
\put(49.875,55){\vector(1,0){9.75}}
\put(49.273,77.49){\vector(1,0){9.75}}
\put(49.096,91.279){\vector(1,0){9.75}}
%\emline(28.416,30.814)(12.666,21.814)
\multiput(28.416,30.814)(-.058988764,-.0337078652){267}{\line(-1,0){.058988764}}
%\end
\put(12.666,21.814){\line(1,0){24.5}}
%\emline(37.166,21.814)(51.291,30.689)
\multiput(37.166,21.814)(.0535037879,.0336174242){264}{\line(1,0){.0535037879}}
%\end
\put(12.666,21.564){\line(0,-1){17.25}}
%\emline(12.666,4.314)(37.041,4.189)
\multiput(12.666,4.314)(6.09375,-.03125){4}{\line(1,0){6.09375}}
%\end
\put(37.041,4.189){\line(0,1){17.5}}
\put(28.416,30.939){\line(0,-1){8.625}}
\put(28.416,21.189){\line(0,-1){8}}
\put(28.416,13.189){\line(1,0){8.25}}
\put(51.291,30.689){\line(0,-1){17.375}}
%\emline(51.291,13.314)(37.166,4.314)
\multiput(51.291,13.314)(-.0529026217,-.0337078652){267}{\line(-1,0){.0529026217}}
%\end
\put(20.791,4.064){\line(0,1){17.75}}
%\emline(28.666,4.189)(36.666,9.564)
\multiput(28.666,4.189)(.05,.03359375){160}{\line(1,0){.05}}
%\end
%\emline(37.541,9.939)(42.166,13.189)
\multiput(37.541,9.939)(.047680412,.033505155){97}{\line(1,0){.047680412}}
%\end
%\emline(12.666,4.064)(20.416,8.814)
\multiput(12.666,4.064)(.054964539,.033687943){141}{\line(1,0){.054964539}}
%\end
%\emline(21.291,9.064)(28.416,13.064)
\multiput(21.291,9.064)(.05987395,.033613445){119}{\line(1,0){.05987395}}
%\end
\put(37.541,13.064){\line(1,0){13.625}}
\put(20.791,4.064){\circle*{1.031}}
\put(41.666,12.814){\circle*{1}}
\put(60.25,16.636){\vector(1,0){6.25}}
%\emline(74.5,12.875)(85.25,31.5)
\multiput(74.5,12.875)(.0336990596,.0583855799){319}{\line(0,1){.0583855799}}
%\end
%\emline(90.125,3.875)(74.625,12.75)
\multiput(90.125,3.875)(-.0587121212,.0336174242){264}{\line(-1,0){.0587121212}}
%\end
%\emline(85.375,31.25)(90.125,4.125)
\multiput(85.375,31.25)(.033687943,-.192375887){141}{\line(0,-1){.192375887}}
%\end
\put(74.5,12.625){\line(1,0){13.75}}
\put(89.125,12.625){\line(1,0){10.625}}
%\emline(90.25,4)(99.625,12.625)
\multiput(90.25,4)(.0366210938,.0336914063){256}{\line(1,0){.0366210938}}
%\end
%\emline(85.25,31.25)(99.375,12.75)
\multiput(85.25,31.25)(.0337112172,-.0441527446){419}{\line(0,-1){.0441527446}}
%\end
%\dashline{1}(22.541,16.564)(44.541,16.564)
\put(22.471,16.494){\line(1,0){.9565}}
\put(24.384,16.494){\line(1,0){.9565}}
\put(26.297,16.494){\line(1,0){.9565}}
\put(28.21,16.494){\line(1,0){.9565}}
\put(30.123,16.494){\line(1,0){.9565}}
\put(32.036,16.494){\line(1,0){.9565}}
\put(33.949,16.494){\line(1,0){.9565}}
\put(35.862,16.494){\line(1,0){.9565}}
\put(37.775,16.494){\line(1,0){.9565}}
\put(39.688,16.494){\line(1,0){.9565}}
\put(41.601,16.494){\line(1,0){.9565}}
\put(43.514,16.494){\line(1,0){.9565}}
%\end
%\dashline{1}(17.75,55.5)(43.125,55.375)
\put(17.68,55.43){\line(1,0){.976}}
\put(19.632,55.42){\line(1,0){.976}}
\put(21.584,55.41){\line(1,0){.976}}
\put(23.535,55.401){\line(1,0){.976}}
\put(25.487,55.391){\line(1,0){.976}}
\put(27.439,55.382){\line(1,0){.976}}
\put(29.391,55.372){\line(1,0){.976}}
\put(31.343,55.362){\line(1,0){.976}}
\put(33.295,55.353){\line(1,0){.976}}
\put(35.247,55.343){\line(1,0){.976}}
\put(37.199,55.334){\line(1,0){.976}}
\put(39.151,55.324){\line(1,0){.976}}
\put(41.103,55.314){\line(1,0){.976}}
%\end
\thicklines
%\dashline{1}(107.125,20)(85.25,31.5)
\multiput(107.055,19.93)(-.0625,.0328571){14}{\line(-1,0){.0625}}
\multiput(105.305,20.85)(-.0625,.0328571){14}{\line(-1,0){.0625}}
\multiput(103.555,21.77)(-.0625,.0328571){14}{\line(-1,0){.0625}}
\multiput(101.805,22.69)(-.0625,.0328571){14}{\line(-1,0){.0625}}
\multiput(100.055,23.61)(-.0625,.0328571){14}{\line(-1,0){.0625}}
\multiput(98.305,24.53)(-.0625,.0328571){14}{\line(-1,0){.0625}}
\multiput(96.555,25.45)(-.0625,.0328571){14}{\line(-1,0){.0625}}
\multiput(94.805,26.37)(-.0625,.0328571){14}{\line(-1,0){.0625}}
\multiput(93.055,27.29)(-.0625,.0328571){14}{\line(-1,0){.0625}}
\multiput(91.305,28.21)(-.0625,.0328571){14}{\line(-1,0){.0625}}
\multiput(89.555,29.13)(-.0625,.0328571){14}{\line(-1,0){.0625}}
\multiput(87.805,30.05)(-.0625,.0328571){14}{\line(-1,0){.0625}}
\multiput(86.055,30.97)(-.0625,.0328571){14}{\line(-1,0){.0625}}
%\end
\thinlines
%\dashline{1}(74.625,12.625)(80.75,15.625)
\multiput(74.555,12.555)(.063802,.03125){12}{\line(1,0){.063802}}
\multiput(76.086,13.305)(.063802,.03125){12}{\line(1,0){.063802}}
\multiput(77.617,14.055)(.063802,.03125){12}{\line(1,0){.063802}}
\multiput(79.148,14.805)(.063802,.03125){12}{\line(1,0){.063802}}
%\end
%\dashline{1}(80.75,15.625)(83.625,17.5)
\multiput(80.68,15.555)(.0513393,.0334821){14}{\line(1,0){.0513393}}
\multiput(82.117,16.492)(.0513393,.0334821){14}{\line(1,0){.0513393}}
%\end
%\dashline{1}(83.625,17.5)(85.375,20.875)
\multiput(83.555,17.43)(.031818,.061364){11}{\line(0,1){.061364}}
\multiput(84.255,18.78)(.031818,.061364){11}{\line(0,1){.061364}}
\multiput(84.955,20.13)(.031818,.061364){11}{\line(0,1){.061364}}
%\end
%\dashline{1}(85.375,20.875)(86,25.375)
\multiput(85.305,20.805)(.03125,.225){4}{\line(0,1){.225}}
\multiput(85.555,22.605)(.03125,.225){4}{\line(0,1){.225}}
\multiput(85.805,24.405)(.03125,.225){4}{\line(0,1){.225}}
%\end
%\dashline{1}(86,25.375)(85.25,31.375)
\put(85.93,25.305){\line(0,1){.8571}}
\put(85.715,27.019){\line(0,1){.8571}}
\put(85.501,28.733){\line(0,1){.8571}}
\put(85.287,30.448){\line(0,1){.8571}}
%\end
\put(10.041,.689){$0$}
\put(10.041,20.814){$0$}
\put(25.916,31.189){$0$}
\put(25.916,12.814){$0$}
\put(42.541,13.189){$2$}
\put(27.916,.689){$2$}
\put(37.041,.689){$3$}
\put(52.916,12.814){$3$}
\put(52.666,31.189){$3$}
\put(38.666,19.689){$3$}
\put(20.116,22.689){$1$}
\put(20.116,.689){$1$}
\put(71.75,11.55){$0$}
\put(100.45,11.55){$2$}
\put(66.809,42.198){$0$}
\put(91.641,52.75){$1$}
\put(75.409,64){$2$}
\put(30.25,44.875){\circle*{1}}
\put(20.75,46.875){$_{01][12]}$}
\put(18.625,59.625){$_{0][012]}$}
\put(38.77,59.625){\makebox(0,0)[cc]{$_{2][012]}$}}
\put(30.395,64.25){\makebox(0,0)[cc]{$_{02]}$}}
\put(36.895,46.875){\makebox(0,0)[cc]{$_{12][01]}$}}
\put(54.125,57.25){$\Upsilon$}
\put(53.523,79.74){$\Upsilon$}
\put(53.346,93.529){$\Upsilon$}
\put(63,18.886){$\Upsilon$}
\put(16.75,41.25){$0$}
\put(16.75,67.25){$0$}
\put(18.13,73.638){$0$}
\put(66.902,73.638){$0$}
\put(67.05,87.662){$0$}
\put(18.75,87.662){$0$}
\put(29.75,41.25){$1$}
\put(39.635,73.949){$1$}
\put(84.614,73.872){$1$}
\put(42.25,41.25){$2$}
\put(42.25,67.25){$2$}
\put(89.858,70.749){$\lambda$}
\put(95.858,57.999){$\lambda$}
\put(88.576,88.06){$\lambda$}
\put(106.435,9.875){$\lambda$}
\put(29.75,56.05){$a$}
\put(79.239,53.354){$_{\Upsilon(\!a)}$}
\put(32.166,17.05){$a$}
\put(81.75,14.55){$_{\Upsilon\!(\!a)}$}
\put(74.5,12.5){\circle*{1}}
\put(90.125,4){\circle*{1}}
\put(99.375,12.625){\circle*{1}}
\put(85.25,31.25){\circle*{1.031}}
\put(17.75,55.45){\circle*{.5}}
\put(43.125,55.45){\circle*{.5}}
\put(22.541,16.5){\circle*{.5}}
\put(44.541,16.5){\circle*{.5}}
\thicklines
%\dashline{1}(89.456,93.817)(67.359,91.696)
\put(89.386,93.747){\line(-1,0){.9607}}
\put(87.464,93.562){\line(-1,0){.9607}}
\put(85.543,93.378){\line(-1,0){.9607}}
\put(83.621,93.193){\line(-1,0){.9607}}
\put(81.7,93.009){\line(-1,0){.9607}}
\put(79.778,92.825){\line(-1,0){.9607}}
\put(77.857,92.64){\line(-1,0){.9607}}
\put(75.935,92.456){\line(-1,0){.9607}}
\put(74.014,92.271){\line(-1,0){.9607}}
\put(72.092,92.087){\line(-1,0){.9607}}
\put(70.171,91.902){\line(-1,0){.9607}}
\put(68.249,91.718){\line(-1,0){.9607}}
%\end
\put(87.798,3.215){\makebox(0,0)[cc]{$1$}}
\put(85.127,33.974){\makebox(0,0)[cc]{$3$}}
\thinlines
%\dashline{1}(67.527,91.549)(76.777,87.555)
\multiput(67.457,91.479)(.076445,-.03301){11}{\line(1,0){.076445}}
\multiput(69.139,90.753)(.076445,-.03301){11}{\line(1,0){.076445}}
\multiput(70.821,90.027)(.076445,-.03301){11}{\line(1,0){.076445}}
\multiput(72.503,89.3)(.076445,-.03301){11}{\line(1,0){.076445}}
\multiput(74.184,88.574)(.076445,-.03301){11}{\line(1,0){.076445}}
\multiput(75.866,87.848)(.076445,-.03301){11}{\line(1,0){.076445}}
%\end
%\dashline{1}(76.777,87.555)(83.084,87.975)
\put(76.707,87.485){\line(1,0){.901}}
\put(78.509,87.605){\line(1,0){.901}}
\put(80.311,87.725){\line(1,0){.901}}
\put(82.113,87.845){\line(1,0){.901}}
%\end
%\dashline{1}(83.084,87.975)(89.181,93.862)
\multiput(83.014,87.905)(.0338694,.0327015){18}{\line(1,0){.0338694}}
\multiput(84.233,89.082)(.0338694,.0327015){18}{\line(1,0){.0338694}}
\multiput(85.452,90.26)(.0338694,.0327015){18}{\line(1,0){.0338694}}
\multiput(86.672,91.437)(.0338694,.0327015){18}{\line(1,0){.0338694}}
\multiput(87.891,92.614)(.0338694,.0327015){18}{\line(1,0){.0338694}}
%\end
\put(18.335,77.885){\line(1,0){21.443}}
\put(67.107,77.464){\line(1,0){19.971}}
%\dashline{1}(67.317,77.254)(75.936,73.05)
\multiput(67.247,77.184)(.065297,-.031852){12}{\line(1,0){.065297}}
\multiput(68.814,76.419)(.065297,-.031852){12}{\line(1,0){.065297}}
\multiput(70.381,75.655)(.065297,-.031852){12}{\line(1,0){.065297}}
\multiput(71.948,74.89)(.065297,-.031852){12}{\line(1,0){.065297}}
\multiput(73.515,74.126)(.065297,-.031852){12}{\line(1,0){.065297}}
\multiput(75.083,73.362)(.065297,-.031852){12}{\line(1,0){.065297}}
%\end
%\dashline{1}(75.936,73.05)(85.397,72.209)
\put(75.866,72.979){\line(1,0){.946}}
\put(77.758,72.811){\line(1,0){.946}}
\put(79.65,72.643){\line(1,0){.946}}
\put(81.542,72.475){\line(1,0){.946}}
\put(83.434,72.307){\line(1,0){.946}}
%\end
%\dashline{1}(85.397,72.209)(92.544,74.942)
\multiput(85.326,72.138)(.088242,.03374){9}{\line(1,0){.088242}}
\multiput(86.915,72.746)(.088242,.03374){9}{\line(1,0){.088242}}
\multiput(88.503,73.353)(.088242,.03374){9}{\line(1,0){.088242}}
\multiput(90.091,73.96)(.088242,.03374){9}{\line(1,0){.088242}}
\multiput(91.68,74.568)(.088242,.03374){9}{\line(1,0){.088242}}
%\end
%\dashline{1}(92.544,74.942)(86.868,77.254)
\multiput(92.474,74.871)(-.081086,.033035){10}{\line(-1,0){.081086}}
\multiput(90.852,75.532)(-.081086,.033035){10}{\line(-1,0){.081086}}
\multiput(89.23,76.193)(-.081086,.033035){10}{\line(-1,0){.081086}}
\multiput(87.609,76.853)(-.081086,.033035){10}{\line(-1,0){.081086}}
%\end
%\emline(90.75,53.75)(69.5,44.25)
\multiput(90.75,53.75)(-.0753546099,-.0336879433){282}{\line(-1,0){.0753546099}}
%\end
%\emline(69.5,44.25)(78.75,64.25)
\multiput(69.5,44.25)(.0336363636,.0727272727){275}{\line(0,1){.0727272727}}
%\end
%\emline(78.75,64.25)(90.75,54)
\multiput(78.75,64.25)(.0394736842,-.0337171053){304}{\line(1,0){.0394736842}}
%\end
%\dashline{1}(69.5,44.25)(87.25,46.25)
\put(69.43,44.18){\line(1,0){.9861}}
\put(71.402,44.402){\line(1,0){.9861}}
\put(73.374,44.624){\line(1,0){.9861}}
\put(75.346,44.846){\line(1,0){.9861}}
\put(77.319,45.069){\line(1,0){.9861}}
\put(79.291,45.291){\line(1,0){.9861}}
\put(81.263,45.513){\line(1,0){.9861}}
\put(83.235,45.735){\line(1,0){.9861}}
\put(85.207,45.957){\line(1,0){.9861}}
%\end
%\dashline{1}(87.25,46.25)(97.25,53.25)
\multiput(87.18,46.18)(.0480769,.0336538){16}{\line(1,0){.0480769}}
\multiput(88.718,47.257)(.0480769,.0336538){16}{\line(1,0){.0480769}}
\multiput(90.257,48.334)(.0480769,.0336538){16}{\line(1,0){.0480769}}
\multiput(91.795,49.41)(.0480769,.0336538){16}{\line(1,0){.0480769}}
\multiput(93.334,50.487)(.0480769,.0336538){16}{\line(1,0){.0480769}}
\multiput(94.872,51.564)(.0480769,.0336538){16}{\line(1,0){.0480769}}
\multiput(96.41,52.641)(.0480769,.0336538){16}{\line(1,0){.0480769}}
%\end
%\dashline{1}(97.25,53.25)(92,62)
\multiput(97.18,53.18)(-.0318182,.0530303){15}{\line(0,1){.0530303}}
\multiput(96.225,54.771)(-.0318182,.0530303){15}{\line(0,1){.0530303}}
\multiput(95.271,56.362)(-.0318182,.0530303){15}{\line(0,1){.0530303}}
\multiput(94.316,57.952)(-.0318182,.0530303){15}{\line(0,1){.0530303}}
\multiput(93.362,59.543)(-.0318182,.0530303){15}{\line(0,1){.0530303}}
\multiput(92.407,61.134)(-.0318182,.0530303){15}{\line(0,1){.0530303}}
%\end
%\dashline{1}(92,62)(79,64.25)
\multiput(91.93,61.93)(-.173333,.03){5}{\line(-1,0){.173333}}
\multiput(90.196,62.23)(-.173333,.03){5}{\line(-1,0){.173333}}
\multiput(88.463,62.53)(-.173333,.03){5}{\line(-1,0){.173333}}
\multiput(86.73,62.83)(-.173333,.03){5}{\line(-1,0){.173333}}
\multiput(84.996,63.13)(-.173333,.03){5}{\line(-1,0){.173333}}
\multiput(83.263,63.43)(-.173333,.03){5}{\line(-1,0){.173333}}
\multiput(81.53,63.73)(-.173333,.03){5}{\line(-1,0){.173333}}
\multiput(79.796,64.03)(-.173333,.03){5}{\line(-1,0){.173333}}
%\end
%\dashline{1}(78.75,64.25)(80,61)
\multiput(78.68,64.18)(.03125,-.08125){8}{\line(0,-1){.08125}}
\multiput(79.18,62.88)(.03125,-.08125){8}{\line(0,-1){.08125}}
\multiput(79.68,61.58)(.03125,-.08125){8}{\line(0,-1){.08125}}
%\end
%\dashline{1}(80,61)(80,58.75)
\put(79.93,60.93){\line(0,-1){.75}}
\put(79.93,59.43){\line(0,-1){.75}}
%\end
%\dashline{1}(80,58.75)(78.75,54.75)
\multiput(79.93,58.68)(-.029762,-.095238){7}{\line(0,-1){.095238}}
\multiput(79.513,57.346)(-.029762,-.095238){7}{\line(0,-1){.095238}}
\multiput(79.096,56.013)(-.029762,-.095238){7}{\line(0,-1){.095238}}
%\end
%\dashline{1}(78.75,54.75)(75.75,50.75)
\multiput(78.68,54.68)(-.0333333,-.0444444){15}{\line(0,-1){.0444444}}
\multiput(77.68,53.346)(-.0333333,-.0444444){15}{\line(0,-1){.0444444}}
\multiput(76.68,52.013)(-.0333333,-.0444444){15}{\line(0,-1){.0444444}}
%\end
%\dashline{1}(75.75,50.5)(69.5,44.5)
\multiput(75.68,50.43)(-.0347222,-.0333333){18}{\line(-1,0){.0347222}}
\multiput(74.43,49.23)(-.0347222,-.0333333){18}{\line(-1,0){.0347222}}
\multiput(73.18,48.03)(-.0347222,-.0333333){18}{\line(-1,0){.0347222}}
\multiput(71.93,46.83)(-.0347222,-.0333333){18}{\line(-1,0){.0347222}}
\multiput(70.68,45.63)(-.0347222,-.0333333){18}{\line(-1,0){.0347222}}
%\end
%\dashline{1}(74.75,12)(83.25,1)
\multiput(74.68,11.93)(.0333333,-.0431373){17}{\line(0,-1){.0431373}}
\multiput(75.813,10.463)(.0333333,-.0431373){17}{\line(0,-1){.0431373}}
\multiput(76.946,8.996)(.0333333,-.0431373){17}{\line(0,-1){.0431373}}
\multiput(78.08,7.53)(.0333333,-.0431373){17}{\line(0,-1){.0431373}}
\multiput(79.213,6.063)(.0333333,-.0431373){17}{\line(0,-1){.0431373}}
\multiput(80.346,4.596)(.0333333,-.0431373){17}{\line(0,-1){.0431373}}
\multiput(81.48,3.13)(.0333333,-.0431373){17}{\line(0,-1){.0431373}}
\multiput(82.613,1.663)(.0333333,-.0431373){17}{\line(0,-1){.0431373}}
%\end
%\dashline{1}(83.25,1)(97.5,1)
\put(83.18,.93){\line(1,0){.95}}
\put(85.08,.93){\line(1,0){.95}}
\put(86.98,.93){\line(1,0){.95}}
\put(88.88,.93){\line(1,0){.95}}
\put(90.78,.93){\line(1,0){.95}}
\put(92.68,.93){\line(1,0){.95}}
\put(94.58,.93){\line(1,0){.95}}
\put(96.48,.93){\line(1,0){.95}}
%\end
%\dashline{1}(97.5,1)(107.5,19.75)
\multiput(97.43,.93)(.0324675,.0608766){14}{\line(0,1){.0608766}}
\multiput(98.339,2.634)(.0324675,.0608766){14}{\line(0,1){.0608766}}
\multiput(99.248,4.339)(.0324675,.0608766){14}{\line(0,1){.0608766}}
\multiput(100.157,6.043)(.0324675,.0608766){14}{\line(0,1){.0608766}}
\multiput(101.066,7.748)(.0324675,.0608766){14}{\line(0,1){.0608766}}
\multiput(101.975,9.452)(.0324675,.0608766){14}{\line(0,1){.0608766}}
\multiput(102.884,11.157)(.0324675,.0608766){14}{\line(0,1){.0608766}}
\multiput(103.793,12.862)(.0324675,.0608766){14}{\line(0,1){.0608766}}
\multiput(104.702,14.566)(.0324675,.0608766){14}{\line(0,1){.0608766}}
\multiput(105.612,16.271)(.0324675,.0608766){14}{\line(0,1){.0608766}}
\multiput(106.521,17.975)(.0324675,.0608766){14}{\line(0,1){.0608766}}
%\end
\put(18.392,77.73){\circle*{1.061}}
\put(39.605,77.73){\circle*{1.061}}
\put(67.183,77.377){\circle*{1.061}}
\put(86.981,77.377){\circle*{1.061}}
\put(78.85,64.295){\circle*{1.061}}
\put(69.304,44.143){\circle*{1.061}}
\put(90.517,53.865){\circle*{1.061}}
\end{picture}

\vspace{0.1in}

\begin{center}
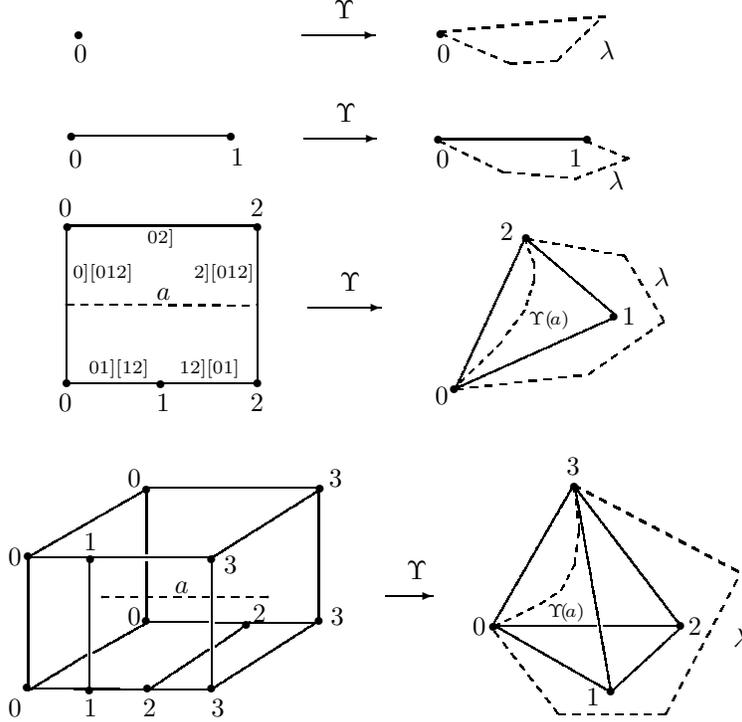
Figure 2. The modelling map $\Upsilon$ for $n_0=0,1,2,3.$
\end{center}

\vspace{0.2in}

\section{Algebraic models for the free loop space and the hat-coHochschild  construction.}

\subsection{Algebraic preliminaries}

We fix a ground commutative ring $\Bbbk$  with unit $1_{\Bbbk}$. All modules are assumed to be  over $\Bbbk.$
We recall some algebraic constructions associated to differential graded coassociative (dgc) coaugmented coalgebras. Recall a dgc coalgebra $(C,d_C, \Delta)$ is \text{coaugmented} if it is equipped with a map of dgc coalgebras $\epsilon: \Bbbk \to C$. Denote $\overline{C}= \text{coker} (\epsilon)$.
 Given a coaugmented dgc coalgebra $(C, d_C, \Delta, \epsilon)$ which is free as a $\Bbbk$-module on each degree, the \textit{cobar construction} of $C$ is the differential graded associative (dga) algebra $(\Omega C,d_{\Omega C})$ defined as follows.
For any $\bar c \in \overline{C}$ write $\Delta(\bar c)= \sum \bar c' \otimes \bar c''$ for the induced coproduct on $\overline{C}$. The underlying algebra of the cobar construction is the tensor algebra \[\Omega C= Ts^{-1}\overline{C}= \Bbbk \oplus s^{-1}\overline{C} \oplus \left(s^{-1} \overline{C}\, \right)^{\otimes 2} \oplus \left(s^{-1} \overline{C}\, \right)^{\otimes 3} \cdots \] and the differential $d_{\Omega C}$ is defined by extending
\[
d_{\Omega C}([\bar c]) =-\left[\, \overline{d_{C}  (c)} \,\right] + \sum (-1)^{|c'|} \left[\, \bar{ c'} \mid \bar {c''} \, \right]
\]
 as a derivation to all of $\Omega C,$ where  an element $s^{-1}\bar c_1 \otimes ... \otimes s^{-1}\bar c_n \in \Omega C$ is denoted by $[\bar c_1|...|\bar c_n].$ The cobar construction defines a functor from the category of coaugmented dgc coalgebras to the category of augmented dga algebras.

The \textit{coHochschild complex} of $C$, as defined in \cite{Hess} and \cite{HPS}, is the dg $\Bbbk$-module $\Lambda C= (C \otimes \Omega C, d_{\Lambda C})$ with differential $d_{\Lambda C}= d_C \otimes 1 + 1 \otimes d_{\Omega C} + \theta_1 + \theta_2$ where

\begin{equation}
\begin{array}{llll}
\theta _1 (v\otimes [ \bar c_1|\dotsb | \bar c_n]) = -\sum (-1)^{|v'|}\,v'\otimes
 [\bar{ v''}| \bar c_1|\!\dotsb\! |\bar c_n],
 \newline $\vspace{1mm}$
 \\
\theta _2 (v\otimes [ \bar c_1|\dotsb | \bar c_n]) = \sum (-1)^{(| v'|+1)
 (|{v''}|+\epsilon^c_n)}\, v''\otimes [ \bar c_1|\!\dotsb \!| \bar c_{n}| \bar{v'}],
\newline $\vspace{1mm}$\\
\hspace{2.8in}
\epsilon^x_n=|x_1|+\cdots +|x_n|+n.
\end{array}
\end{equation}

Adams showed in \cite{Adams} that the cobar construction on the reduced dgc coalgebra of singular chains on a simply connected topological space is quasi-isomorphic as a dga algebra to the singular chains on the based loop space.
It was then explained in \cite{Rivera- Zeinalian} and \cite{RS} why the same construction works for a path connected, possibly non-simply connected, space. The cobar construction is not invariant under quasi-isomorphisms of dgc coalgebras so in order to obtain the correct model for the chains on the based loop space one must be careful when choosing the dgc coalgebra model for the underlying space.  Moreover, it was explained in \cite{RS} how for any path connected simplicial set $X$, the chains on the necklical set $\widehat{\mathbf{\Omega}}X$ may be understood as a model generalizing Adams' cobar construction on the dgc coalgebra of simplicial chains on $X$, similar to the \textit{extended cobar construction} of \cite{Hess- Tonks}. In \cite{HPS} and \cite{saneFREE} it was shown that the coHochschild complex of the dg coalgebra of chains on a simply connected simplicial set $X$ calculates the homology of the free loop space on $|X|$. Below, we generalize this fact: we explain how the chains associated to the closed necklical set $\widehat{\mathbf{\Lambda}}X$ yields a complex which generalizes the coHochschild complex on the simplicial chains of $X$ and calculates the homology of the free loop space on $|X|$ for a path connected, possibly non-simply connected, simplicial set $X$. The resulting chain complex is small and suitable for computations.

\subsection{The hat-coHochschild construction} Let $(X,x_0)$ be a pointed simplicial set and denote by $(C_\ast(X),d_{C},\Delta)$  the simplicial chain dgc coalgebra with Alexander-Whitney coproduct. The \textit{hat-coHochschild} complex $\widehat{\Lambda} C_\ast(X)$ of $C_\ast(X)$ is in fact obtained from  the coHochschild complex $\Lambda C_\ast(X)$
 by replacing the cobar construction $\Omega C_\ast(X)$ by the hat-cobar construction $\widehat{\Omega} C_\ast(X).$

 First, recall the definition of the \textit{hat-cobar construction}. Consider coaugmented the dgc coalgebra $(C_\ast (Z(X)), d_C,\Delta, \epsilon)$, where $\epsilon$ is determined by the choice of fixed point $x_0$. Let $C_{\ast > 0} (X_0)$ denote the sub dgc coalgebra of $C_\ast (Z(X))$ generated by simplices of positive degree on the subsimplicial set of $Z(X)$ generated by the vertices $X_0=Z(X)_0$, so all generators of $C_{\ast > 0} (X_0)$ are degenerate simplices having degenerate faces.
We may truncate $d_C$ and $\Delta$ to obtain a new coaugmented dg coalgebra $A:=(A_\ast(X), d_A,\Delta', \epsilon)$  where
 \[ A_\ast(X)=C_\ast (Z(X)) / C_{\ast > 0} (X_0), \,\,
d_A=d_C-\partial_0-(-1)^{n}\partial_n: A_n\rightarrow  A_{n-1},\]
and $\Delta'$ is $\Delta$ without the primitive  term.
Let $(\Omega A,d_{\Omega A})$ be the cobar construction of $A$ and
define for $n>0$, $\Omega'_n A\subset \Omega _nA$
to be the submodule generated by monomials $[\bar a_1|\cdots| \bar a_k]\in \Omega_n A,\,k\geq 1,  $   where each $a_i$ is a simplex in $Z(X)$ representing a generator of $A$ such that
 $\min a_1=\max a_k= x_0 $ and
 $\max a_i=\min a_{i+1}$
for all $i;$ $\Omega^{\prime}_0 A=\Bbbk$, and $\Omega^{\prime}_n A =0 $ for $n<0$. Then $\Omega'A$ inherits the structure of a dg algebra. In particular, $\Omega' A= \Omega A$ when $X_0=\{x_0\}$.
 Define the \emph{hat-cobar construction}  $\widehat{\Omega}C_\ast(X)$ of the dgc coalgebra $C_\ast(X)$ as
\[  \widehat{\Omega}C_\ast(X)=\Omega' A /\sim ,  \]
where $\sim$ is generated by
\[
[\bar a_1|...|\bar a_{i-1}|\bar a_i|\bar a_{i+1}|\bar a_{i+2}|...|\bar a_k ]\sim [\bar a_1|...|\bar a_{i-1}|\bar a_{i+2}|...|\bar a_k]
\ \  \text{whenever}\ \  a_{i+1}=a_i^{op};
\]
in particular, $[\,\bar a_i|\bar a_{i+1}]\sim 1_{\Bbbk}.$ Note that the dg algebra of chains on the necklical set $\widehat{\mathbf{\Omega}} X$    coincides with the hat-cobar construction $\widehat{\Omega}C_\ast(X).$

The \textit{hat-coHochschild complex} is defined as $\widehat{\Lambda} C_\ast(X)=C_\ast(X)\otimes \widehat{\Omega} C_\ast(X)$ with  differential
        $d_{\widehat{\Lambda}C}= d_C\otimes 1 + 1\otimes d_{ \widehat{\Omega} C}+ \theta _1+\theta _2,$
 where $\theta_1$ and $\theta_2$ are defined as in 4.1.

  The homology of  $\widehat{\Lambda} C_\ast(X) $ is called the \textit{hat-coHochschild homology}
 of
$ C_\ast(X)$ and is denoted by $\widehat{HH}_*( C_\ast(X)).$

The chain complex $(C_\ast(\widehat{\mathbf{\Lambda}}X),d)$ of  the closed necklical set $\widehat{\mathbf{\Lambda}} X$ is

\[C_\ast(\widehat{\mathbf{\Lambda}} X)=C'_\ast(\widehat{\mathbf{\Lambda}} X)/C'_\ast(D(e)),\]
where $ C'_\ast(\widehat{\mathbf{\Lambda}} X)$ is
the free $\Bbbk$-module  generated by
the set
$\widehat{\mathbf{\Lambda}} X$
and
$D(e)\subset \widehat{\mathbf{\Omega}} X \subset \widehat{\mathbf{\Lambda}} X$ denotes the set  of degeneracies arising from the unit $e\in \widehat{\mathbf{\Omega}} X,$
and the differential $d=\{d_n\}_{n\geq 1}$ with
$d_n:   C_n(\widehat{\mathbf{\Lambda}}X)  \rightarrow  C_{n-1}(\widehat{\mathbf{\Lambda}}X)    $
given by
 $d_n=\bigoplus_{\substack{{n=n_0+r}\\{n_0,r\geq 0}}}\, d_{n_0,r}$ so that the component
\[  d_{n_0,r}= \sum _{i=1}^{n} (-1)^{i}(d_i^0-d_i^1)+\sum _{i=2}^{n_0} (-1)^{(i-1)n} d^2_i
   \]
acts on $   C_\ast(\widehat{\mathbf{\Lambda}}_{n_0,r}X) .$
 We have a straightforward
\begin{theorem}\label{hat-coHoch} For a simplicial set $X$
the chain complex  $C_\ast(\widehat{\mathbf{\Lambda}} X)$    coincides with the hat-coHochschild complex $\widehat{\Lambda}C_\ast(X).$
\end{theorem}
In particular,
for a $1$-reduced $X$ (e.g., $X=\operatorname{Sing}^1(Y,y)$ the simplicial set consisting of all singular simplices in a topological space $Y$ which collapse edges to a fixed point $y \in Y$), the hat-cobar construction $\widehat{\Omega} C_\ast(X)$ coincides with the Adams' cobar construction  $\Omega C_\ast(X)$ of  the dg coalgebra $C_\ast(X)$  and, consequently, the hat-coHochschild construction $\widehat{\Lambda} C_\ast(X)$ coincides with the standard coHochschild construction  $\Lambda C_\ast(X).$
Thus  we obtain (compare \cite{saneFREE})
\begin{theorem}\label{coHoch} For a 1-reduced simplicial set $X$
the chain complex  $C_\ast({\mathbf{\Lambda}} X)$    coincides with the coHochschild complex ${\Lambda}C_\ast(X).$
\end{theorem}

It follows directly from Theorems \ref{freeloopmodel} and \ref{hat-coHoch}  that for a path connected simplicial set $X$ we have an isomorphism  $\widehat{HH}_*( C_\ast(X)) \cong H_*(\Lambda Y)$ for $Y=|X|$. Moreover,  from the homotopy invariance of the free loop space we have the following direct
\begin{corollary}
If $f_\ast:C_\ast(X)\rightarrow C_\ast(X')$ is induced by a weak equivalence $f: X\rightarrow X',$
then $\widehat{\Lambda}f_\ast:\widehat{\Lambda}C_\ast(X) \rightarrow \widehat{\Lambda}C_\ast(X') $ is a quasi-isomorphism.
\end{corollary}

\subsection{Hochschild chain models for the free loop space}
Recall that for any augmented dga algebra $(A,d_A, \cdot, \mu: A \to \Bbbk)$ over $\Bbbk$, the \textit{Hochschild complex} of $A$ is defined as the dg module $Hoch(A)= (BA \otimes A, d_{Hoch}=1 \otimes d_{A} + d_{BA} \otimes 1 + \theta^1+\theta^2 )$
where
\[BA:= T^csA= \Bbbk \oplus \left(s \overline{A}\,\right) \oplus \left(s \overline{A}\,\right)^{\otimes 2} \oplus \left(s \overline{A}\,\right)^{\otimes 3} \oplus \cdots ,
\]
$\overline{A}= \text{ker}(\mu: A \to \Bbbk)$, $s$ denotes the shift by $+1$ functor, and $d_{BA}=d_1+d_2$ denotes the \textit{bar construction differential} with
\[\ \  d_{1}[\,\bar a_{1}|\dotsb|\bar a_{n}]=-\sum_{i=1}^{n} (-1)^{\epsilon^a_{i-1}}
[\,\bar a_{1}|\dotsb|\overline{d_{A}(a_{i})}\,|\dotsb|\bar a_{n}],\]
\[d_{2} [\,\bar a_{1}|\dotsb|\bar a_{n}]= -\sum_{i=1}^{n-1} (-1)^{\epsilon^a_{i}}
[\,\bar a_{1}|\dotsb|\overline{a_{i}a_{i+1}}\,|\dotsb|\bar a_{n}],
\]
and
\begin{equation*}
\begin{array}{lll}
\theta ^1([\bar a_1|\dotsb |\bar a_n]\otimes u) = (-1)^{|a_1|(|u|+\epsilon^a_n+|a_1|+1)} [\bar a_2|\dotsb |\bar a_n]\otimes ua_1,\newline $\vspace{1mm}$ \\

\theta ^2 ( [\bar a_1|\dotsb |\bar a_n]\otimes u) =
-(-1)^{\epsilon^a_{n-1}}  [\bar a_1|\dotsb |\bar a_{n-1}]\otimes a_n u.
\end{array}
 \end{equation*}
 where an element $s \bar a_1 \otimes \dotsb \otimes s\bar a_n \in BA$ is denoted by  $[\bar a_1 | \dotsb | \bar a_n]$.
Recall that $(BA, d_{BA})$ is a dgc coalgebra equipped with deconcatenation coproduct.

The following result was stated in \cite{Hess}, however, we outline a proof  not relying on the comparison theorem for spectral sequences of twisted tensor products, which assumes certain hypotheses.

\begin{proposition} Let $C$ be a connected dgc coalgebra. Then there is a quasi-isomorphism of dg modules $coHoch(C) \simeq Hoch(\Omega C)$.
\end{proposition}

\begin{proof}
Consider the canonical projection map $\rho:B\Omega C\to C$. The map $\rho$ is not a coalgebra map; however $\rho$ induces a chain map \[\phi: Hoch(\Omega C) \to coHoch(C)\] defined as follows.
For any $x\otimes u\in B\Omega C \otimes \Omega C $  define $\phi$ by (compare  \cite{J-M.hoh})
\[\phi(x\otimes u)=\left\{
\begin{array}{llll}
1\otimes u,  && x= [\  ],\\

( s\otimes 1 )\chi(a\otimes  u), && x =[\bar a],\\

0, &&  x= [\bar a_1|\cdots |\bar a_n],& n\geq 2,
\end{array}
\right.
\]
where $\chi:\Omega C \otimes \Omega C\rightarrow  s^{-1}\overline{C}\otimes \Omega C$
is a map given for $a\otimes u\in \Omega C\otimes \Omega C$ with  $a=\bar a_1\cdots \bar a_n,\,a_i\in C,$     by
\[
\chi(a\otimes u)=\left\{
\begin{array}{llll}
         0,            &&  a= 1,\\

\bar a_1\otimes u,      && n=1,\vspace{1mm}\\

\underset{1\leq i\leq n}{\sum}(-1)^{\varepsilon}
 \bar a_i \otimes \bar a_{i+1}\cdots \bar a_{n}\, u\, \bar a _1\cdots \bar a_{i-1},& &    n\geq 2,
\end{array}
\right.
\]
$\hspace{1.8in} \varepsilon=(|a_{i-1}|+\cdots +|a_n|+n+i)(|u|+|a_1|+\cdots +|a_i|+i).$
The above construction is part of the following general fact shown in \cite{Hess}, which we do not need in full generality for our purposes:  $\rho$ is actually a morphism of coalgebras up to strong homotopy compatible with twisting cochains, and these kind of morphisms induce chain maps between the more general notion of Hochschild complexes associated to a twisting cochain.

There is a ``local" chain contraction $s: \text{Ker} \rho \rightarrow \text{Ker} \rho $ defined for $x:=[\bar x_1|...|\bar x_n]\in B\Omega C$ with $x_1=\bar c_1\cdots \bar c_k\in \Omega C$ by
\[
s(x)=\left\{
  \begin{array}{lll}
    0,                                              &  k=1,\\
    \left[\,\bar {\bar {c}}_1|\, \overline{\bar{c}_2\cdots \bar{c}_k}\,| \bar{x}_2|...|\bar{x}_n\right], & k>1;
  \end{array}
\right.
\]
this means that
for an element $a\in (\text{Ker} \rho,d_{\Omega BC})$   there exists a positive integer $m_a$ such that   $m_a^{th}$-composition   $(sd+ds-id)^{m_a}(a)=0.$ This follows from the conilpotency of the coproduct $\Delta$ on $\overline{C}$, namely, from the fact that for any $\bar c \in \overline{C} $, $\Delta^{k_c}(\bar c)=0$ for some positive integer $k_c$. Extend $s$ on  $\text{Ker} \phi$ as $s'=s\otimes id.$ Then $s'$ will similarly satisfy $(s'd_{Hoch} + d_{Hoch}s' -id )^{n_y}(y)=0$ for some positive integer $n_y$. It follows that every cycle in $\text{Ker} \phi$ bounds, so the complex $(\text{Ker} \phi, d_{Hoch})$ is acyclic.
Then the result follows from the long exact homology sequence of
$0\to \text{Ker} \phi\to B\Omega C \otimes \Omega C\overset{\phi}{\longrightarrow} C\otimes \Omega C\to 0$.
\end{proof}
\begin{remark}
\normalfont
In fact, the local contraction $s$ above may be used to obtain a chain homotopy between $\eta \circ \rho$ and $id_{B\Omega C}$ where $\eta: C \to B \Omega C$ is the (strict) dgc coalgebra map given by
\[\eta(x)= \{ [x] \} + \sum \{ [x'] | [x'']\} + \sum \{ [x'] | [x''] | [x'''] \} + \cdots .\]
Similarly, $s'$ may be used to obtain a chain homotopy between the chain maps $(\eta \otimes id_C) \circ \phi $ and $id_{Hoch(\Omega C)}.$
\end{remark}
For any space $Y$ denote $S^N_*(Y,y):= C^N_*(\text{Sing}(Y,y))$ the dgc coalgebra of normalized singular chains on $Y$ with vertices on $y \in Y$. It was shown in \cite{Rivera- Zeinalian} that for any path connected space $Y$ there is a quasi-isomorphism of dga algebras

\[ S^N_*(\Omega Y) \simeq \Omega S^N_*(Y,y).
\]

As explained in Remark 2 of \cite{RS}, there is no need to add formal inverses (as in 3.4) to obtain the correct dga algebra on the right hand side since $\text{Sing}(Y,y)$ is a Kan complex and every $1$-simplex is invertible up to homotopy.
In fact, the same exact argument as in the proof of Theorem \ref{freeloopmodel} follows through to show that $| \mathbf{\Lambda} \text{Sing} (Y,y) |$ is homotopy equivalent to $\Lambda Y$. Moreover, the normalized chain complex associated to the closed necklical set $\mathbf{\Lambda} \text{Sing}(Y,y)$ is isomorphic as a dg module to $coHoch(S^N_*(Y,y))$, the coHochschild complex on the dgc coalgebra $S^N_*(Y,y)$. It follows that $coHoch(S^N_*(Y,y))$ is a dg module quasi-isomorphic to the normalized singular chains on the free loop space $S^N_*(\Lambda Y)$. Hence, by Proposition 3, we obtain quasi-isomorphisms
\[ Hoch( \Omega S^N_*(Y,y) ) \simeq coHoch( S^N_*(Y,y) )  \simeq S_*(\Lambda Y).
\]
Finally, by the invariance of the Hochschild chain complex with respect to quasi-isomorphisms of dga algebras we may deduce the following classical result proved in \cite{Goodwillie} using different methods
\begin{corollary}
For any path connected topological space $Y$ there is a quasi-\linebreak
isomorphism of dg modules $Hoch(S^N_*(\Omega Y) ) \simeq S^N_*(\Lambda Y)$.
\end{corollary}

\vspace{0.2in}

\bibliographystyle{plain}

\end{document}